\documentclass[11pt,reqno]{amsart}
\usepackage{graphicx,amssymb,mathrsfs,amsmath,amsfonts,appendix}
\usepackage{subfigure}
\usepackage{amsthm}
\usepackage{paralist}
\usepackage{enumitem}
\usepackage[utf8]{inputenc}
\usepackage{comment}
\usepackage{amsmath}
\usepackage{mathtools}
\usepackage{esint}
\usepackage[colorlinks=true, pdfstartview=FitV, linkcolor=blue, citecolor=blue, urlcolor=blue]{hyperref}
\numberwithin{equation}{section}

\newtheorem{proposition}{Proposition}[section]
\newtheorem{theorem}{Theorem}[section]

\newtheorem{remark}{Remark}[section]

\usepackage[colorinlistoftodos]{todonotes}

\newcommand{\xd}{\textrm{d}}

\renewcommand{\div}[0]{\operatorname{div}}
\newcommand{\ep}[0]{\varepsilon}
\newcommand{\R}{\mathbb{R}}
 
\newcommand{\embed}{\hookrightarrow}
\def\wstarto{\stackrel{*}{\rightharpoonup}}
\newcommand{\mU}{\mathcal U}

\allowdisplaybreaks

\title[Local asymptotics for a tumor growth model]
{Local asymptotics and optimal control 
for a viscous Cahn-Hilliard-Reaction-Diffusion model for tumor growth}
\author[Elisa Davoli]{Elisa Davoli}
\address{Institute for Analysis and Scientific Computing, Vienna University of Technology, Wiedner Hauptstrasse 8-10, 1040 Vienna, Austria}
\email{elisa.davoli@tuwien.ac.at}
\author[Elisabetta Rocca]{Elisabetta Rocca}
\address{Dipartimento di Matematica, Università di Pavia
and IMATI - C.N.R., Via Ferrata 5, 27100 Pavia, Italy}
\email{elisabetta.rocca@unipv.it}
\author[Luca Scarpa]{Luca Scarpa}
\address{Department of Mathematics, Politecnico di Milano,
Via E.~Bonardi 9, 20133 Milano, Italy}
\email{luca.scarpa@polimi.it}
\author[Lara Trussardi]{Lara Trussardi}
\address{Institute of Mathematics and Scientific Computing, University of Graz, Heinrichstraße 36, 8010 Graz (Austria)}
\email{lara.trussardi@uni-graz.at}

\begin{document}

\begin{abstract}
In this paper we study nonlocal-to-local asymptotics for a tumor-growth model coupling a viscous Cahn-Hilliard equation describing the tumor proportion with a reaction-diffusion equation for the nutrient phase parameter. 
First, we prove that solutions to the nonlocal Cahn-Hilliard system converge, as the nonlocality parameter tends to zero, to solutions to its local counterpart. Second, we provide first-order optimality conditions for an optimal control problem on the local model, accounting also for chemotaxis, and both for regular or singular potentials, without any additional regularity assumptions on the solution operator. The proof is based on an approximation of the local control problem by means of suitable nonlocal ones, and on proving nonlocal-to-local convergence both for the corresponding dual systems and for the associated first-order optimality conditions.
\end{abstract}

\maketitle

\tableofcontents

\section{Introduction}
\label{sec:intro}
We study here diffuse interface tumor-growth models characterized by viscous nonlocal and local Cahn-Hilliard systems coupled with reaction-diffusion equations. This formulation was introduced first in \cite{hawk2} and successively studied in many papers in the literature both from  the well-posedness~\cite{ col-gil-hil,  col-gil-roc-spr, col-gil-roc-spr2,frig-grass-roc} and the optimal control~\cite{col-gil-roc-spr-opt, col-sig-spr2, col-sig-spr} point of view.\\

Nonlocal Cahn-Hilliard type systems are currently the subject of an intense research activity. Among the vast literature we single out the contributions \cite{ab-bos-grass-NLCH, bat-han-NLCH, gal-gior-grass-NLCH, gal-grass-NLCH, han-NLCH} and the references therein. The first proof of existence of solutions for nonlocal Cahn-Hilliard equations with singular kernels with lower than $W^{1,1}$-regularity was carried out by some authors of this paper in \cite{DRST,DST}, where nonlocal-to-local convergence of the associated solutions was also established both in the Dirichlet and in the viscous homogeneous Neumann setting. The case of $W^{1,1}$ kernels was the subject of \cite{DST2}. In the same work, the authors provided also a connection with the seminal works of Sandier and Serfaty in \cite{sand-serf} on the evolutionary convergence of gradient flows with applications to Ginzburg-Landau functionals (see also \cite{lin, sand-serf}, as well as the study of a simplified tumor growth model in \cite{roc-sc}). Nonlocal-to-local convergence results have then been obtained in more general settings such as coupled systems \cite{AT} and degenerate mobilities \cite{ca-el-sk-DEG,el-sk-DEG}. Precise rates of convergence have been provided in \cite{AH}.
The strict separation property was analyzed in \cite{gal-gior-gras-SEP,poiatti-SEP}, whereas the case of degenerate mobility was the subject of \cite{frig-gal-gras-DEG}.\\

In the present contribution we focus instead on nonlocal-to-local asymptotics for Cahn-Hilliard tumor-growth models. To the Authors' knowledge, this is the first paper in the literature in this direction. In order to describe our results in details, we need to introduce some notation.\\

We denote by $\varphi\in(-1,1)$ the difference in volume fractions between the tumor cells ($\varphi=1$) and the healthy cells ($\varphi=-1$), and by $\sigma\in(0,1)$ the concentration of the nutrient. We then consider the following
local problem in $(0,T)\times\Omega$, where
$T>0$ is a fixed final time and $\Omega$
is a domain in $\mathbb R^d$:
\begin{align}
&\label{eq1:l}
\partial_t\varphi  -\div(m(\varphi)\nabla\mu) = R-h(\varphi)u\,,\\ 
&\label{eq1':l}
R=P(\varphi)(\sigma +\chi(1-\varphi)-\mu)
\,,\\ 
&\label{eq2:l}
\mu  =
\tau\partial_t\varphi
-\Delta\varphi +\psi'(\varphi)-\chi \sigma\,,
&\\ 
&\label{eq3:l}
\partial_t\sigma -\div[n(\varphi)\nabla(\sigma+\chi (1-\varphi))]=-R+w\,.
\end{align}
Here $m$ and $n$ are the (bounded from below and above) mobilities, 
$\tau>0$ is the viscosity coefficient, and $\chi\geq 0$ is the chemotaxis coefficient.
The potential $\psi$ is generally a  double-well potential having two global minima in correspondence of the pure phases: in particular, the case of polynomial potentials of the form
\begin{align}
	&\label{eq:reg-pot} \psi(s) = \frac{1}{4} (1 - s^2)^2, \quad s \in \R, 
\end{align}
can be included in our analysis. 
The term $R$ on the right-hand side
of the Cahn-Hilliard equation is
motivated from linear phenomenological laws
for chemical reactions  (see \cite{hawk2}), and takes into account the proliferation and death of tumoral cells. As mentioned in \cite{hawk2}, by means of proper numerical simulations, it is possible to show that the predominant term in \eqref{eq1':l} is the first one $P(\varphi)\sigma$, while the remaining terms are relevant only for large values of $\chi$. Moreover, the choice of $R$ automatically guarantees the thermodynamical consistency of the models, that is, the fact that the energy decreases along solutions in the case in which no source terms are present. The proliferation function  $P$  calibrates the strength of the reaction terms and is usually  taken as follows: 
\begin{equation} 
\label{eq:ex-P}
P(s) = \begin{cases}
	P_0 & \text{if } s \le -1, \\
	\frac{P_1 - P_0}{2}(s+1) + P_0 & \text{if } -1 < s <1, \\
	P_1 & \text{if } s \ge 1,
\end{cases} \end{equation}
where $P_0$ and $P_1$ are suitable positive constants, so that  the proliferation is faster in the tumor phase. 
There are possible different choices in the literature for the term $R$ on the right hand side, cf., e.g., \cite{garcke-lam,garcke-et-al,  garcke-lam-roc, SS21} and the references therein.
Note that, for technical reasons, in this paper we will be forced to assume that $P$ is strictly positive, albeit still possibly very close to zero.
This will be needed in proving the existence of strong solutions to \eqref{eq1:nl}--\eqref{eq5:nl} in the case of a singular potential (cf.~also \cite{CRW,fornoni} for similar assumptions). 

The external sources $u$ and $w$ are the controls that we are going to use in our optimal control problem to direct the evolution of the system towards suitable targets. 
In the bio-medical applications, they are interpreted as external therapies acting on the tumor: $u$ can represent a radiotherapy (distributed through the function $h$) acting on the proliferation of the tumor cells, whereas $w$ can be seen as a chemotherapy acting on the tumor through drugs or antiangiogenic therapy (cf., e.g., \cite{CGLRR}).

Finally, we complement the problem with 
homogeneous Neumann boundary conditions and 
given initial data:
\begin{align}
&\label{eq4:l}
\partial_{\bf n} \varphi=
\partial_{\bf n} \mu =
\partial_{\bf n} \sigma = 
0\qquad&\text{on } (0,T)\times \partial\Omega\,,\\
&\label{eq5:l}
\varphi(0)=\varphi_0\,, \quad
\sigma(0)=\sigma_0 \qquad&\text{in } \Omega\,.
\end{align}

The corresponding nonlocal problem, studied (in some cases with extra regularizing terms), e.g., in \cite{fornoni, frig-lam-roc, frig-lam-sig}, is given by
the following PDE system in $(0,T)\times\Omega$
\begin{align}
&\label{eq1:nl}
\partial_t\varphi_\ep  -\div(m(\varphi_\ep)\nabla\mu_\ep) 
= R_\ep - h(\varphi_\ep)u_\ep,\\ 
&\label{eq1':nl}
R_\ep=P(\varphi_\ep)(\sigma_\ep +\chi(1-\varphi_\ep)-\mu_\ep)
,\\ 
&\label{eq2:nl}
\mu_\ep  =
\tau\partial_t\varphi_\ep +
(J_\ep * 1)\varphi_\ep -J_\ep *\varphi_\ep +\psi'(\varphi_\ep)-\chi \sigma_\ep
,\\ 
&\label{eq3:nl}
\partial_t\sigma_\ep -\div[n(\varphi_\ep)\nabla(\sigma_\ep+\chi (1-\varphi_\ep))]=-R_\ep + w_\ep,
&
\end{align}
coupled with the following initial and boundary conditions
\begin{align}\label{eq4:nl}
&\partial_{\bf n} \mu_\ep =
n(\varphi_\ep)\partial_{\bf n} (\sigma_\ep-\chi\varphi_\ep) = 
0\qquad&\text{on } (0,T)\times \partial\Omega\,,\\
&\label{eq5:nl}
\varphi_\ep(0)=\varphi_{0,\ep}\,, \quad
\sigma_\ep(0)=\sigma_{0,\ep} 
\qquad&\text{in } \Omega\,.
\end{align}

The free energy associated to the local system \eqref{eq1:l}--\eqref{eq5:l} reads as
\begin{align*}
\mathcal{E}(\varphi, \sigma) := \int_{\Omega} \left(\frac{|\nabla \varphi|^{2}}{2}+  \psi(\varphi) + \frac{1}{2} |\sigma|^{2} + \chi \sigma (1-\varphi)\right) \xd x,
\end{align*}
where the first two terms form the well-known Modica--Mortola energy, leading to phase separation and surface tension effects.

The energetics of the nonlocal variant \eqref{eq1:nl}--\eqref{eq5:nl} of \eqref{eq1:l}--\eqref{eq5:l} are encoded by the following nonlocal free energy where the Dirichlet contribution in $\mathcal{E}$ is replaced by the following nonlocal counterpart:
\begin{align*}
\int_{\Omega} \int_{\Omega} \frac{1}{4} J_\ep(x-y)(\varphi(x) - \varphi(y))^{2} \xd x\, \xd y,
\end{align*}
where $J_\ep$ is a symmetric interaction kernel defined on $\Omega \times \Omega$.\\

In many different biological models, nonlocal interactions have been used to describe competition for space and degradation \cite{SRLC}, spatial redistribution \cite{Borsi, Lee}, and cell-to-cell adhesion \cite{Armstrong, Chaplin, Gerisch} and, at least in a formal way, the
nonlocal dynamics approach the local ones when the family of interaction kernels $J_\ep$ concentrates around the origin.

Our main contribution is a rigorous characterization of such nonlocal-to-local asymptotics both for the tumor-growth model introduced above, as well as for an associated control problem, which we describe below.  
\smallskip

\noindent(CP) \textit{Minimise the cost functional}
\begin{equation}
\label{eq:J}
	\begin{split}
		\mathcal{J}(\varphi, \sigma, u, w) & = \, \frac{\alpha_{\Omega}}{2} \int_{\Omega} |\varphi(T) - \varphi_{\Omega}|^2 + \frac{\alpha_Q}{2} \int_{0}^{T} \int_{\Omega} |\varphi - \varphi_Q|^2 
        + \frac{\beta_Q}{2} \int_{0}^{T} \int_{\Omega} |\sigma - \sigma_Q|^2 \\ 
		& \quad + \frac{\alpha_u}{2} \int_{0}^{T} \int_{\Omega} |u|^2  + \frac{\beta_w}{2} \int_{0}^{T} \int_{\Omega} |w|^2\,,
	\end{split}
\end{equation}
\textit{subject to the control constraints}
\begin{equation} 
\label{eq:uad}
	\begin{split}
	& u \in \mathcal{U}_{ad} := \{ u \in L^{\infty}(Q_T) \cap H^1(0,T;L^2(\Omega)) \mid u_{\text{min}} \le u \le u_{\text{max}} \text{ a.e. in } Q_T, \\ 
	& \qquad \qquad \qquad \|u\|_{H^1(0,T;L^2(\Omega))} \le M \}, \\
	& w \in \mathcal{V}_{ad} := \{ w \in L^{\infty}(Q_T) \mid w_{\text{min}} \le w \le w_{\text{max}} \text{ a.e. in } Q_T \},
	\end{split}
\end{equation}
\textit{and to the nonlocal state system \eqref{eq1:nl}-\eqref{eq5:nl}.} 
\smallskip


The cost functional $\mathcal{J}$ is a standard tracking type cost functional, where the constants $\alpha_\Omega, \alpha_Q, \beta_Q, \alpha_u, \beta_w$ are non-negative parameters that can be used to weigh the target functions $\varphi_\Omega$ (a final target for the tumor distribution), $\varphi_Q$ (a desired tumor evolution), as well as $\sigma_Q$ (desired evolution for the nutrient). 
Finally, the last two terms in the cost functional penalise large use of radiotherapy or drugs, which could still harm vital organs of the patient. From the point of view of biological applications, the most natural choice for $\mathcal{U}_{ad}$ would be $\mathcal{U}_{ad}=\{u\in L^\infty(Q_T)|u_{\rm min}\leq u\leq u_{max}\,\text{a.e. in }Q_T\}$. The higher regularity requirement $u\in H^1(0,T;L^2(\Omega))$ in \eqref{eq:uad} is needed in order to achieve well-posedness owing to the fact that $u$ acts as source term in the Cahn-Hilliard equation (see also \cite[Theorems 3.5 and 3.7]{fornoni}).\\

Our main result is twofold. First, in Theorem~\ref{th:conv} we show that solutions to the nonlocal Cahn-Hilliard tumor-growth systems above converge, as $\ep\searrow 0$, to solutions of its corresponding local counterpart. Second, starting from the nonlocal optimal control problem described above, in Theorems~\ref{th:opt_adap} and \ref{th:conv_dual} we pass to the limit as $\ep\to 0$ in the associated optimality conditions and adjoint system, recently identified in \cite{fornoni}. As a result, in Theorem \ref{th:foc_local} we derive first order necessary optimality conditions for the control problem driven by the same cost functional  $\mathcal{J}$ as in \eqref{eq:J} and encoded by the local system \eqref{eq1:l}--\eqref{eq5:l} without imposing any extra regularity or differentiability assumptions on the local solution operator. Indeed,  one of the main advantages of the nonlocal setting is
that more regularity results  are usually available and so, having at disposal rigorous
nonlocal-to-local convergence results,  gives us  the opportunity to approximate
solutions to local phase-field systems with the solutions to the corresponding nonlocal
ones. For completeness we point out that from a modeling point of view, it would be meaningful to incorporate in $\mathcal{J}$ a further term enforcing a final target for the nutrient. Handling such mathematical term is, currently, still out of reach: In fact, in order to characterize the asymptotics of the final-time behavior of the nutrient as $\ep\to 0$ in the adjoint problem, the presence of the chemotaxis coefficient $\chi$ would require stronger nonlocal-to-local convergences compared to the ones which we deduce in this contribution, cf. Section~\ref{sec:n-l opt cont}.\\

The proof of Theorem~\ref{th:conv} relies on the establishment of uniform a-priori bounds for the solutions to the nonlocal system, as well as on a combination of the theory for maximal monotone graphs with some interpolation lemmas established in \cite{DST2}, and on a further compactness arguments for the chemotaxis, chemical-reaction term and proliferation function. Theorem~\ref{th:opt_adap} exploits essentially the variational structure of the control problems and is based on choosing the right competitors in the local and nonlocal settings. The starting point for identifying first-order optimality conditions for the local control problem in Theorems~\ref{th:opt_adap}, \ref{th:conv_dual}, and \ref{th:foc_local} are the results in \cite{fornoni}. Relying on such characterizations of the optimality conditions and of the dual systems for the nonlocal case, by choosing suitable test functions for the nonlocal dual system we establish $\ep$-independent uniform estimates which eventually allow us to pass to the limit and characterize the local asymptotics for the control problem.\\

The paper is organized as follows. In Section~\ref{sec:prel} we specify our mathematical setting and recall a few preliminary results from \cite{fornoni} on the well-posedness of the system for $\ep>0$ fixed.
In Section \ref{sec:state} we state and prove Theorem~\ref{th:conv}. Section~\ref{sec:opt} is entirely devoted to the local and nonlocal control problems: after specifying our working assumptions, we recall the characterization of the dual system and of the first-order optimality conditions in \cite{fornoni} and we state and prove Theorem~\ref{th:opt_adap}. Eventually, Section~\ref{sec:n-l opt cont} focuses on the passage to the limit in the nonlocal optimality conditions and on the identification of the corresponding local ones, and consists of the statements and proofs of Theorems~\ref{th:conv_dual} and \ref{th:foc_local}.

\section{Setting and preliminary results}
\label{sec:prel}
In the paper $\Omega\subset \mathbb{R}^d$, $d=2,3$, is a smooth bounded domain with sufficiently smooth boundary and $T>0$ is a fixed final time.
We set $Q:=(0,T)\times\Omega$ and $Q_t:=(0,t)\times\Omega$ for all $t\in(0,T)$. We define the functional spaces 
\[
H:=L^2(\Omega)\,,\qquad V:=H^1(\Omega)\,, \qquad
W:=\{v\in H^2(\Omega): \;
\partial_{\bf n}v=0 \;\text{a.e.~on}\; \partial\Omega\}\,,
\]
endowed with their natural norms.
We identify $H$ with its dual, as usual, so that we have the following 
dense, continuous, and compact inclusions:
\[
  W\embed V \embed H \embed V^*\embed W^*\,.
\]
The duality pairing between $V^*$ and 
$V$ will be denoted by the symbol 
$\langle\cdot,\cdot\rangle$.
For any $v\in V^*$, we will use the classical notation 
$v_\Omega:=\frac1{|\Omega|}
\langle v, 1\rangle$
for the mean value of $v$.
The variational formulation of the
negative Laplacian with homogeneous Neumann conditions
is denoted by 
\[
B:V\to V^*\,, \qquad \langle Bv_1, v_2\rangle_{V^*, V}:=
\int_\Omega\nabla v_1\cdot\nabla v_2\,\xd x\,,
\quad v_1,v_2\in V\,.
\]

\smallskip

We assume the following setting,
as in~\cite{fornoni}.
\begin{enumerate}[start=1,label={{\bf A\arabic*}}]
    \item
    For every $\ep>0$ we define
    the convolution kernel
    $J_\ep:\mathbb{R}^d\to\mathbb{R}$ as 
    \begin{equation}\label{eq:kernel}
      J_\ep(z) :=\frac{\rho_\ep (|z|)}{\ep^{2-\alpha}|z|^\alpha}\,,
      \quad z\in\mathbb{R}^3\,,
    \end{equation}
    where $\alpha\in[0,d-2)$ is fixed,
    and
    \[
      \rho_\ep(r):=
      \frac{1}{\ep^d}\rho(r/\ep)\,,
      \quad r\in[0,+\infty)\,.
    \]
    Here, $\rho:[0,+\infty)\to[0,+\infty)$ is a given function with compact support, of class $C^2$, and such that
    \begin{align}
    \label{rho0}
    &\int_0^{+\infty}
    r^{d-1-\alpha}|\rho''(r)|\,\xd r
    <+\infty\,,\\
    \label{rho1}
    &c_{\alpha,d}:=\int_0^{+\infty}
    r^{d-1-\alpha}|\rho'(r)|\,\xd r
    <+\infty\,,\\
    \label{rho2}
    &\int_0^{+\infty}
    r^{d+1-\alpha}\rho(r)\,\xd r=
    \frac2{C_{dim}}\,, \qquad
    C_{dim}:=\int_{S^{d-1}}|\sigma\cdot
    e_1|^2\,\xd \mathcal H^{d-1}(\sigma)
    \,,
    \end{align}
    where $S^{d-1}$ is the unit sphere in $\mathbb R^d$.
    In this setting, for every $\ep>0$
we define
\[
  (J_\ep*v)(x):=
  \int_\Omega J_\ep(x-y)v(y)\,\xd y\,,
  \quad\forall\,v\in L^1(\Omega)\,.
\]
From \eqref{eq:kernel} and
\eqref{rho1}--\eqref{rho2}
it follows that (see \cite{DST2, DST})
\[
  z\mapsto J_\ep(|z|) \in W^{2,1}_{\text{loc}}
  (\mathbb{R}^d)\,,
\]
so that we can define the nonlocal energy and its respective differential
\[
    E_\ep:H\to[0,+\infty)\,,
    \qquad B_\ep:H\to H
\]
as
\begin{align*}
    &E_\ep(v):=\frac14
    \int_{\Omega\times\Omega}
    J_\ep(x-y)|v(x)-v(y)|^2
    \,\xd x\,\xd y\,,
    \quad v\in H\,,\\        &B_\ep(v):=
    (J_\ep*1)v-J_\ep*v\,,
    \quad v\in H\,.
\end{align*}
Let us also recall that $E_\ep\in C^1(H)$ with 
$DE_\ep=B_\ep$ (see again \cite{DST2, DST}).
 \item The potential $\psi:\mathbb{R}\to[0,+\infty)$ is
    of class $C^2$, and there are 
    constants $C_\psi, c_\psi, M_\psi>0$
    such that,
    for every $r\in\mathbb{R}$ 
    there holds
    \begin{align*}
       &\psi' (r)r\geq c_\psi |r|^2 -c_\psi^{-1}\,,\\
       &|\psi''(r)|\leq M_\psi(1+|r|^2)
       \,, \qquad
       \psi''(r)\geq-C_\psi
       \,,\\
       &|\psi'(r)|\leq M_\psi(1+\psi(r))\,,
    \end{align*}
    which implies, by possibly
    renominating $c_\psi$, that
    \[
    \psi (r)\geq c_\psi |r|^2 -c_\psi^{-1} \quad\forall\,r
    \in\mathbb{R}\,.
    \]
    Defining 
  \[
  0<\ep_0^2< \frac{c_{\alpha,d}|S^{d-1}|}{C_\psi}\,,
  \]
   by virtue of 
  \cite[Lem.~3.1]{DST2},  
  for every $\ep\in(0,\ep_0)$ there holds 
  \[
  \psi''(r) + (J_\ep*1)(x) \geq 
  \frac{c_{\alpha,d}|S^{d-1}|}{\ep_0^2}-C_\psi>0
  \quad\forall\,r\in\mathbb R\,,\quad\text{for a.e.~}x\in\Omega\,.
  \]
    \item The viscosity and chemotaxis coefficients
    satisfy $\tau>0$ and $\chi\in[0,\sqrt{c_\psi})$. 
    \item The mobilities
    satisfy
    $m,n\in C^0(\mathbb{R})$, and there
    exist $M_*,M^*>0$ such that 
    \[
    M_*\leq m(r),n(r)\leq M^*
    \quad\forall\,r\in\mathbb{R}\,.
    \]
    \item The
    proliferation function satisfies $P\in C^0(\mathbb{R})$ 
    and 
    there exists $C_P>0$ such that
    \[
    0\leq P(r) \leq C_P(1+|r|^4)
    \quad\forall\,r\in\mathbb{R}\,.
    \]
    \item 
    The interpolation function $h\in C^0(\mathbb{R})\cap L^{\infty}(\mathbb{R})$ is 
    nonnegative.
    \item The initial data of the local problem satisfy
    \begin{equation}
    \label{init_l}
    \varphi_0\in V\,, \qquad
    \psi(\varphi_0)\in L^1(\Omega)\,,
    \qquad
    \sigma_0\in H\,.
    \end{equation}
    The initial data of the $\ep$-nonlocal problem satisfy
    \begin{align}
    \label{init_nl1}
    &\varphi_{0,\ep}\in V\,, \quad
    \psi(\varphi_{0,\ep})\in L^1(\Omega)\,,
    \quad
    \sigma_{0,\ep}\in H
    \qquad\forall\,\ep\in(0,\ep_0)\,,\\
    \label{init_nl2}
    &(\varphi_{0,\ep}, \sigma_{0,\ep})
    \to
    (\varphi_0,\sigma_0) \quad\text{in } H\times H
    \qquad\text{as }
    \ep\searrow0\,,\\
    \label{init_nl3}
    &\sup_{\ep\in(0,\ep_0)}\left(
    E_\ep(\varphi_{0,\ep})
    +\|\psi(\varphi_{0,\ep})\|_{L^1(\Omega)}
    +\tau\|\varphi_{0,\ep}\|^6_{L^6(\Omega)}\right)
    <+\infty\,.
    \end{align}
\end{enumerate}

\smallskip

We point out that the polynomial potentials in \eqref{eq:reg-pot} fall within our analysis.
In this setting, the nonlocal problem at $\ep>0$ fixed has been widely studied in \cite{frig-lam-roc},
also in the setting of degenerate mobilities but without viscous 
regularisation. In our setting, the 
viscosity contribution allows us to 
refine the assumptions and exploit the results of \cite{fornoni} instead.
We recall here 
the weak existence result 
for the nonlocal problem proven in \cite[Thm. 2.3]{fornoni}.


\begin{theorem}
[Well-posedness: $\ep>0$]
  \label{th:wp_nl}
  Assume {\bf A1}--{\bf A7}
  and let $\ep\in(0,\ep_0)$. Then, 
  for every
  \begin{equation}\label{uw}
      (u,w)\in [L^\infty(Q)\cap H^1(0,T; H)]\times L^\infty(Q)\,,
  \end{equation}
  there exists a  triplet 
  $(\varphi_\ep, \mu_\ep, \sigma_\ep)$, with 
  \begin{align}
  \label{nl1}
  &\varphi_\ep\in 
  H^1(0,T; H)\cap L^\infty(0,T; V)\,,\\
  \label{nl2}
  &\sigma_\ep\in 
  H^1(0,T; V^*)\cap L^2(0,T; V)\,,\\
  \label{nl3}
  &\mu_\ep:=
  \tau\varphi_\ep+
  B_\ep(\varphi_\ep)+
  \psi'(\varphi_\ep)-\chi\sigma_\ep
  \in L^2(0,T; V)\,,\\
  \label{nl4}
  &R_\ep:=P(\varphi_\ep)(\sigma_\ep+\chi(1-\varphi_\ep)-\mu_\ep)
  \in L^2(0,T; L^{6/5}(\Omega))\,,
  \end{align}
  such that
  \begin{equation}
      \label{nl5}
      \varphi_\ep(0)=\varphi_{0,\ep}\,,
      \quad
      \sigma_\ep(0)=\sigma_{0,\ep}\,,
  \end{equation}
  and
  \begin{align}
  \label{nl6}
      &\langle\partial_t\varphi_\ep,
      \zeta\rangle
      +\int_\Omega
      m(\varphi_\ep)\nabla\mu_\ep\cdot
      \nabla\zeta \xd x= 
      \int_\Omega
      R_\ep\zeta\,\xd x
      -\int_\Omega
      h(\varphi_\ep)u\zeta\,\xd x\,,\\
  \label{nl7}
      &\langle\partial_t\sigma_\ep,
      \zeta\rangle +
      \int_\Omega n(\varphi_\ep)
      \nabla(\sigma_\ep+
      \chi(1-\varphi_\ep))\cdot
      \nabla\zeta\,\xd x=
      -\int_\Omega
      R_\ep\zeta\,\xd x
      +\int_\Omega
      w\zeta\,\xd x\,,
  \end{align}
  almost everywhere in $(0,T)$,
  for every $\zeta\in V$. 
\end{theorem}
Under suitably refined hypothesis it is possible to 
obtain strong well-posedness for the nonlocal problem.
In order to state the result in \cite{fornoni}, we introduce also the 
following strengthened assumptions.
\begin{enumerate}[start=1,label={{\bf B\arabic*}}]
  \item The potential satisfies $\psi\in C^4(\mathbb R)$ and 
  \[
  \lim_{r\to\pm\infty}(\psi'(r)-\chi^2r)=\pm\infty\,.
  \]
  \item The coefficient $\ep_0>0$ satisfies 
  \[
  0<\ep_0^2< \frac{c_{\alpha,d}|S^{d-1}|}{\chi^2+C_\psi}\,,
  \]
  so that for every $\ep\in(0,\ep_0)$ there holds 
  \[
  \psi''(r) + J_\ep*1(x) \geq 
  \frac{c_{\alpha,d}|S^{d-1}|}{\ep_0^2}-C_\psi>\chi^2
  \quad\forall\,r\in\mathbb R\,,\quad\text{for a.e.~}x\in\Omega\,.
  \]
  \item The mobilities are constant, namely $m\equiv 1$ and $n\equiv1$.

  \item The proliferation and truncation functions
  satisfy $P,h\in C^1(\mathbb R)\cap W^{1,\infty}(\mathbb R)$
  and there exists a constant $P_0>0$ such that 
  \[
  P(r)\geq P_0 \quad\forall\,r\in\mathbb R\,.
  \]
  Moreover, we set $C_h:=\|h\|_{W^{1,\infty}(\mathbb R)}$. 
  \item The initial data of the $\ep$-nonlocal problem satisfy
    \begin{equation}\label{init_nl1_str}
    \varphi_{0,\ep}\in H^2(\Omega)\,, \quad
    \quad
    \sigma_{0,\ep}\in V\cap L^\infty(\Omega)
    \qquad\forall\,\ep\in(0,\ep_0)\,.
    \end{equation}
\end{enumerate}
Examples of proliferation functions satisfying our assumptions are those described in \eqref{eq:ex-P}.
In this refined setting, the nonlocal problem is well-posed 
in the strong sense, as specified in the following result, following directly from \cite[Thm~3.5]{fornoni}.


  \begin{theorem}\label{wp_nl_strong}
  Assume {\bf A1}--{\bf A7}, {\bf B1}--{\bf B5}, 
  and let $\ep\in(0,\ep_0)$. Then, for every $(u,w)$
  satisfying \eqref{uw}, 
  there exists a unique triplet $(\varphi_\ep, \mu_\ep, \sigma_\ep)$,
  with 
\begin{align}
  \label{nl1_str}
  &\varphi_\ep\in 
  W^{1,\infty}(0,T; H)\cap H^1(0,T; H^2(\Omega))\,,\\
  \label{nl2_str}
  &\sigma_\ep\in 
  H^1(0,T; H)\cap L^\infty(0,T; V)\cap L^2(0,T; W)\,,\\
  \label{nl3_str}
  &\mu_\ep=
  \tau\varphi_\ep+
  B_\ep(\varphi_\ep)+
  \psi'(\varphi_\ep)-\chi\sigma_\ep
  \in L^\infty(0,T; W)\,,\\
  \label{nl4}
  &R_\ep:=P(\varphi_\ep)(\sigma_\ep+\chi(1-\varphi_\ep)-\mu_\ep)
  \in L^2(0,T; L^{6}(\Omega))\,,
  \end{align}
  such that
  \begin{equation}
      \label{nl5_str}
      \varphi_\ep(0)=\varphi_{0,\ep}\,,
      \quad
      \sigma_\ep(0)=\sigma_{0,\ep}\,,
  \end{equation}
  and
  \begin{align}
  \label{nl6_str}
      &\partial_t\varphi_\ep
      -\Delta\mu_\ep= 
      R_\ep-
      h(\varphi_\ep)u \quad\text{a.e.~in } Q\,,\\
  \label{nl7_str}
      &\partial_t\sigma_\ep
      -\Delta(\sigma_\ep+
      \chi(1-\varphi_\ep))=
      -R_\ep +w \quad\text{a.e.~in } Q\,.
  \end{align}
  Moreover, 
  there exists a constant $C_\ep>0$,
  depending only on the structure of the problem,
  such that, for every pair of
  initial data 
  $\{(\varphi_{0,\ep}^i, \sigma_{0,\ep}^i)\}_{i=1,2}$
  satisfying \eqref{init_nl1} and \eqref{init_nl1_str}
  and for every pair 
  $\{(u_i, w_i)\}_{i=1,2}$ satisfying \eqref{uw},
  then the corresponding strong solutions
  $\{(\varphi_\ep^i, \mu_\ep^i, \sigma_\ep^i)\}_{i=1,2}$
  to \eqref{nl1_str}--\eqref{nl7_str} satisfy
  \begin{align}
      \nonumber
      &\|\varphi_\ep^1
      -\varphi_\ep^2\|_{H^1(0,T; H)\cap L^\infty(0,T; H^2(\Omega))}
      +\|\mu_\ep^1-\mu_\ep^2\|_{L^2(0,T;W)}
      +\|\sigma_\ep^1-\sigma_\ep^2\|_{H^1(0,T; H)\cap L^2(0,T; W)}\\
      \label{cont_nl}
      &\leq
      C_\ep\left(
      \|\varphi_{0,\ep}^1
      -\varphi_{0,\ep}^2\|_{H^2(\Omega)}
      +\|\sigma_{0,\ep}^1
      -\sigma_{0,\ep}^2\|_V
      +\|u_1-u_2\|_{L^2(0,T; H)} + \|w_1-w_2\|_{L^2(0,T; H)}
      \right)\,.
  \end{align}
\end{theorem}

\section{Nonlocal-to-local convergence of the state systems}
\label{sec:state}
The first main result of this paper is the convergence of the nonlocal system \eqref{eq1:nl}--\eqref{eq3:nl}
to the local one \eqref{eq1:l}--\eqref{eq3:l}, as $\ep\searrow0$. This shows, as a byproduct, existence of weak solutions for the local problem.

\begin{theorem}[Convergence of the state systems]
  \label{th:conv}
  Assume {\bf A1}--{\bf A7}, let 
  $(u,w)$ and $\{(u_\ep,w_\ep)\}_{\ep\in(0,\ep_0)}$ satisfy 
  \eqref{uw} and 
  \begin{align*}
  u_\ep\wstarto u \quad&\text{in } L^\infty(Q)\cap H^1(0,T; H)\,,\\
  w_\ep\wstarto w \quad&\text{in } L^\infty(Q)\,.
  \end{align*}
  Let also $\{(\varphi_\ep, \mu_\ep, \sigma_\ep)\}_{\ep\in(0,\ep_0)}$ 
  be a family of weak solutions to the nonlocal system
  in the sense of Theorem~\ref{th:wp_nl}
  corresponding to the family $\{(u_\ep,w_\ep)\}_{\ep\in(0,\ep_0)}$.
  Then, there exists a triplet 
  $(\varphi,\mu,\sigma)$, with
  \begin{align}
  \label{loc_first}
  &\varphi\in 
  H^1(0,T; H)\cap L^\infty(0,T; V)
  \cap L^2(0,T; W)\,,\\
  &\sigma \in H^1(0,T; V^*)
  \cap L^2(0,T; V)\,,\\
  &\mu=
  \tau\partial_t\varphi
  -\Delta\varphi+
  \psi'(\varphi)-\chi\sigma
  \in L^2(0,T; V)\,,\\
  &R:=P(\varphi)(\sigma+\chi(1-\varphi)-\mu)
  \in L^2(0,T; L^{6/5}(\Omega))\,,
  \end{align}
  such that 
  \begin{equation}
  \varphi(0)=\varphi_{0}\,, \qquad
  \sigma(0)=\sigma_{0}\,,
  \end{equation}
  and
  \begin{align}
      &\nonumber\langle\partial_t\varphi,
      \zeta\rangle
      +\int_\Omega
      m(\varphi)\nabla\mu\cdot
      \nabla\zeta\, \xd x = 
      \int_\Omega
      R\zeta\,\xd x
      -\int_\Omega
      h(\varphi)u\zeta\,\xd x,\\
      &\langle\partial_t\sigma,
      \zeta\rangle +
      \int_\Omega n(\varphi)
      \nabla(\sigma+
      \chi(1-\varphi))\cdot
      \nabla\zeta\,\xd x=
      -\int_\Omega
      R\zeta\,\xd x
      +\int_\Omega
      w\zeta\,\xd x,
      \label{loc_last}
  \end{align}
  almost everywhere in $(0,T)$,
  for every $\zeta\in V$. Moreover,
  as $\ep\searrow0$ there holds
  \begin{align}
  \label{conv_loc_first}
      &\varphi_\ep \to\varphi
      \quad\text{in } C^0([0,T]; H)\,,\\
      &\varphi_\ep\wstarto\varphi
      \quad\text{in }
      H^1(0,T; H)\cap 
      L^\infty(0,T; L^6(\Omega))\,,\\
      &\mu_\ep\rightharpoonup\mu 
      \quad\text{in } L^2(0,T; V)\,,\\
      &\sigma_\ep\to\sigma 
      \quad\text{in }
      L^2(0,T; H)\cap C^0([0,T]; V^*)\,,\\
      &\sigma_\ep\wstarto\sigma 
      \quad\text{in }
      H^1(0,T; V^*)\cap L^\infty(0,T; H)\,,\\
      &\sigma_\ep+\chi(1-\varphi_\ep)\rightharpoonup
      \sigma+\chi(1-\varphi) \quad\text{in }
      L^2(0,T; V)\,.
  \label{conv_loc_last}
  \end{align}
  Additionally, if also {\bf B3}--{\bf B4} are satisfied, then 
  $(\varphi,\mu,\sigma)$
  is the unique weak solution to the local system in the 
  sense of \eqref{loc_first}--\eqref{loc_last}, there holds 
  \[
  \varphi\in L^2(0,T; H^3(\Omega))\,,
  \]
  and we also have
  \[
  \sigma_\ep(t)\to\sigma(t) \quad\text{in } H \quad\forall\,t\in[0,T]\,.
  \]
  Moreover, for every pair of initial data 
  $\{(\varphi_{0}^i, \sigma_{0}^i)\}_{i=1,2}$
  satisfying \eqref{init_l} and for every pair 
  $\{(u_i, w_i)\}_{i=1,2}$ satisfying \eqref{uw},
  there exists a constant $C>0$ such that
  the respective weak solutions $(\varphi_1,\mu_1,\sigma_1)$ and 
  $(\varphi_2,\mu_2,\sigma_2)$ to the local system in the 
  sense of \eqref{loc_first}--\eqref{loc_last} 
  satisfy
  \begin{align}
      \nonumber
      &\|\varphi_1-\varphi_2\|_{H^1(0,T; H)\cap L^\infty(0,T; V)}
      +\|\mu_1-\mu_2\|_{L^2(0,T; V)}
      +\|\sigma_1-\sigma_2\|_{L^\infty(0,T; H)\cap L^2(0,T; V)}\\
      \label{cont_dep_loc}
      &\leq C\left(\|\varphi_{0}^1-\varphi_{0}^2\|_V
      +\|\sigma_{0}^1-\sigma_{0}^2\|_H+
      \|u_1-u_2\|_{L^2(0,T; L^{6/5}(\Omega))}
      +\|w_1-w_2\|_{L^2(0,T; L^{6/5}(\Omega))}\right).
  \end{align}
\end{theorem}

The rest of this section is devoted to the proof of Theorem~\ref{th:conv}.
\subsection{Uniform estimates}
We test \eqref{eq1:nl} by $\mu_\ep$,
\eqref{eq2:nl} by $-\partial_t\varphi_\ep$,
\eqref{eq3:nl} by $\sigma_\ep+\chi(1-\varphi_\ep)$, and sum all contributions.
Integrating by parts in time, we obtain
\begin{align*}
    &\int_{Q_t}m(\varphi_\ep)|\nabla\mu_\ep|^2\,\xd x\xd s+
    \tau\int_{Q_t}|\partial_t\varphi_\ep|^2\,\xd x \xd s+
    E_\ep(\varphi_\ep(t))+\int_\Omega\psi(\varphi_\ep(t))\,\xd x\\
    &\qquad+\frac{1}{2}\int_\Omega |\sigma_\ep(t)|^2\,\xd x
    +\int_{Q_t}n(\varphi_\ep)|\nabla(\sigma_\ep+\chi(1-\varphi_\ep)|^2\,\xd x \xd s\\
    &= E_\ep(\varphi_{0,\ep})+
    \int_\Omega\psi(\varphi_{0,\ep})\,\xd x
    +\frac{1}{2}\int_\Omega |\sigma_{0,\ep}|^2\,\xd x
    +\int_{Q_t}(R_\ep-h(\varphi_\ep)u_\ep) \mu_\ep\,\xd x \xd s\\ 
    &\qquad+\int_{Q_t}(-R_\ep+w_\ep)
    [\sigma_\ep+\chi(1-\varphi_\ep)]\,\xd x \xd s
    +\chi\int_{Q_t}\sigma_\ep\partial_t \varphi_\ep\,\xd x \xd s-\chi\int_{Q_t}\partial_t \sigma_\ep(1-\varphi_\ep)\,\xd x \xd s\\
    &= E_\ep(\varphi_{0,\ep})
    +\int_\Omega\psi(\varphi_{0,\ep})\,\xd x
    +\frac{1}{2}\int_\Omega |\sigma_{0,\ep}|^2\,\xd x 
    -\int_{Q_t}P_\ep(\varphi)
    |\sigma_\ep+\chi(1-\varphi_\ep)-\mu_\ep|^2\,\xd x \xd s\\
    &\qquad-\int_{Q_t}h(\varphi_\ep)u_\ep \mu_\ep\,\xd x \xd s
    +\int_{Q_t}w_\ep
    [\sigma_\ep+\chi(1-\varphi_\ep)]\,\xd x \xd s\\
    &\qquad-\chi\int_\Omega \sigma_\ep(t)(1-\varphi_\ep(t))\,\xd x+
    \chi\int_\Omega \sigma_{0,\ep}(1-\varphi_{0,\ep})\,\xd x.
\end{align*}
Now, on the left-hand side we use
assumptions {\bf A2}, while to estimate the right-hand side
we rely on assumption {\bf A7}, the fact that $(u_\ep)_\ep$ and
$(w_\ep)_\ep$ are uniformly bounded in $L^\infty(Q)$, and
Young inequality. We infer 
\begin{align*}
    &\int_{Q_t}m(\varphi_\ep)|\nabla\mu_\ep|^2\,\xd x\xd s
    +\tau\int_{Q_t}|\partial_t\varphi_\ep|^2\,\xd x\xd s+
    E_\ep(\varphi_\ep(t))+
    c_\psi\int_\Omega|\varphi_\ep(t)|^2\,\xd x
    \\
    &\qquad
    +\frac{1}{2}\int_\Omega |\sigma_\ep(t)|^2\,\xd x
    +\int_{Q_t}n(\varphi_\ep)|\nabla(\sigma_\ep+\chi(1-\varphi_\ep)|^2
    \,\xd x\xd s\\
    &\qquad+\int_{Q_t}P(\varphi_\ep)|\sigma_\ep+
    \chi(1-\varphi_\ep)-\mu_\ep|^2\,\xd x\xd s\\
    &\leq \frac{|\Omega|}{c_\psi} 
    +E_\ep(\varphi_{0,\ep})+
    \int_\Omega\psi(\varphi_{0,\ep})\,\xd x
    +\frac{1}{2}\int_\Omega |\sigma_{0,\ep}|^2\,\xd x
    +M\int_{Q_t}(|\mu_\ep|+|\varphi_\ep|^2+|\sigma_\ep^2|)\,\xd x\xd s\\
    &\qquad+ M 
    -\chi\int_\Omega \sigma_\ep(t)(1-\varphi_\ep(t))\,\xd x+
    \chi\int_\Omega \sigma_{0,\ep}(1-\varphi_{0,\ep})\,\xd x\\
    &\leq \frac{|\Omega|}{c_\psi} 
    +E_\ep(\varphi_{0,\ep})+
    \int_\Omega\psi(\varphi_{0,\ep})\,\xd x
    +\int_\Omega |\sigma_{0,\ep}|^2\,\xd x
    +M\int_{Q_t}(|\mu_\ep|+|\varphi_\ep|^2+|\sigma_\ep^2|)\,\xd x\xd s\\
    &\qquad+ M 
    +\chi^2\int_\Omega|\varphi_{0,\ep}|^2\,\xd x
    +\chi^2|\Omega|
    +\frac{\delta}{2}\int_\Omega|\sigma_\ep(t)|^2\,\xd x
    +\frac{\chi^2}{\delta}\int_\Omega|\varphi_\ep(t)|^2\,\xd x
    +\frac{\chi^2}{\delta}|\Omega|\,,
\end{align*}
where $M>0$ and $\delta>0$ are constants and independent of $\ep$. 
Now, since $\chi^2<c_\psi$ by assumption
{\bf A3}, it is possible 
to choose $\delta$ sufficiently small such that the two terms 
on the right-hand side corresponding to $\sigma_\ep(t)$ and
$\varphi_\ep(t)$ can be incorporated in the left-hand side.
Indeed, it is enough to choose $\delta>0$ such that 
\[
  \frac\delta2<\frac12 \qquad\text{and}\qquad
  \frac{\chi^2}\delta<c_\psi\,,
\]
which is possible if and only if 
$\frac{\chi^2}{c_\psi}<1$.

Recalling also assumptions {\bf A4} and {\bf A7}, and using the 
Poincar\'e inequality on the term in $\mu_\ep$ on the right-hand side,
we infer the existence of a positive constant 
$M$, independent of $\ep$, such that
\begin{align*}
    &\int_{Q_t}|\nabla\mu_\ep|^2\,\xd x\xd s
    +\int_{Q_t}|\partial_t\varphi_\ep|^2\,\xd x\xd s+
    E_\ep(\varphi_\ep(t))+
    \int_\Omega|\varphi_\ep(t)|^2\,\xd x
    +\int_\Omega |\sigma_\ep(t)|^2\,\xd x
    \\
    &\qquad
    +\int_{Q_t}|\nabla(\sigma_\ep+\chi(1-\varphi_\ep)|^2
    \,\xd x\xd s
    +\int_{Q_t}P(\varphi_\ep)|\sigma_\ep+
    \chi(1-\varphi_\ep)-\mu_\ep|^2\,\xd x\xd s\\
    &\leq M
    \left(1+\int_0^t|(\mu_\ep)_\Omega|\,\xd s+
    \int_{Q_t}(|\varphi_\ep|^2+|\sigma_\ep^2|)\,\xd x\xd s\right)\,.
\end{align*}
Now, by the symmetry of the kernel $J_\ep$, arguing by comparison in equation \eqref{eq2:nl} we find 
\[
  (\mu_\ep)_\Omega = \tau(\partial_t\varphi_\ep)_\Omega
  +(\psi'(\varphi_\ep))_\Omega - \chi(\sigma_\ep)_\Omega\,,
\]
so that assumption {\bf A2} and the Young inequality 
ensure that, for every $\delta>0$ there exists a suitable constant $M_\delta>0$ such that,
\[
  \int_0^t|(\mu_\ep)_\Omega|\,\xd s
  \leq \delta\tau^2\int_{Q_t}|\partial_t\varphi_\ep|^2+
  M_\delta\int_{Q_t}
  \left(1+|\sigma_\ep|^2 + |\psi(\varphi_\ep)|\right)\,\xd x\xd s\,.
\]
Putting everything together, choosing $\delta$ small enough, 
and using the Gronwall lemma
yields then
the following estimates:
\begin{align}
    \label{eq:est1}
    \Vert E_\ep(\varphi_\ep)\Vert_{L^\infty(0,T)}+\Vert \psi(\varphi_\ep)\Vert_{L^\infty(0,T;L^1(\Omega)}
    +\|\varphi_\ep\|_{H^1(0,T; H)}^2&\leq M\\
    \label{eq:est2}
    \Vert\nabla\mu_\ep\Vert_{L^2(0,T; H)}
    &\leq M\\
    \label{eq:est3}
    \Vert\nabla(\sigma_\ep +\chi(1-\varphi_\ep))\Vert_{L^2(0,T; H)}
    &\leq M\\
    \label{eq:est4}
    \Vert\sigma_\ep\Vert_{L^\infty(0,T;H)} &\leq M\,.
\end{align}
It follows in particular, using again 
assumption {\bf A2}, 
the Poincar\'e-Wirtinger 
inequality, and comparison in \eqref{eq2:nl}, that
\begin{align}
    \label{eq:est5}
  \Vert\varphi_\ep\Vert_{L^\infty(0,T;H)}
  +\|\mu_\ep\|_{L^2(0,T; V)}
  +\|\sigma_\ep 
    +\chi(1-\varphi_\ep)\|_{L^2(0,T; V)}
  &\leq M\,.
\end{align}

At this point, we test equation \eqref{eq2:nl} by $B_\ep(\varphi_\ep)$ and integrate in time, obtaining
\begin{align*}
    &\tau E_\ep(\varphi_\ep(t)) + 
    \int_{Q_t}|B_\ep(\varphi_\ep)|^2\,\xd x \xd s +
    \int_{Q_t}\psi'(\varphi_\ep)B_\ep(\varphi_\ep)\,\xd x \xd s\\
    &\qquad=\tau E_\ep(\varphi_{0,\ep})
    + \int_{Q_t}(\mu_\ep+\chi\sigma_\ep)B_\ep(\varphi_\ep)\,\xd x \xd s.
\end{align*}
Recalling that 
$r\mapsto \psi'(r) + C_\psi r$ is nondecreasing on $\mathbb{R}$
thanks to assumption {\bf A2} on $\psi''$, we have
\begin{align*}
    &\int_{Q_t}\psi'(\varphi_\ep)B_\ep(\varphi_\ep)\,\xd x \xd s\\
    &=\int_0^t\int_{\Omega\times\Omega}
    J_\ep(x-y)(\psi'(\varphi_\ep(s,x))-\psi'(\varphi_\ep(s,y)))
    (\varphi_\ep(s,x)-\varphi_\ep(s,y))\,\xd x\,\xd y\,\xd s\\
    &\geq-C_\psi\int_0^t\int_{\Omega\times\Omega}
    J_\ep(x-y)|\varphi_\ep(s,x)-\varphi_\ep(s,y)|^2\,\xd x\,\xd y\,\xd s
    =-4C_\psi \int_0^t E_\ep(\varphi_\ep(s))\,\xd s\,. 
\end{align*}
Putting this information together, using Young inequality,
assumption {\bf A7},
and estimates \eqref{eq:est1}, \eqref{eq:est4} and \eqref{eq:est5}, 
we deduce that
\begin{align*}
    \tau E_\ep(\varphi_\ep(t)) + 
    \int_{Q_t}|B_\ep(\varphi_\ep)|^2\,\xd x \xd s
    \leq M + \frac12\int_{Q_t}|B_\ep(\varphi_\ep)|^2\,\xd x \xd s
\end{align*}
for a certain constant $M>0$ independent of $\ep$.
It follows that 
\begin{equation}
    \label{eq:est8}
    \Vert B_\ep(\varphi_\ep)\Vert_{L^2(0,T;H)}\leq M\,,
\end{equation}
and hence, by comparison in \eqref{eq2:nl},
\begin{equation}
    \label{eq:est9}
    \Vert\psi'(\varphi_\ep)\Vert_{L^2(0,T;H)}\leq M\,.
\end{equation}

We now show a further integrability estimate
on $(\varphi_\ep)_\ep$. For every $p\in[2,6]$,
we test equation \eqref{eq2:nl}
by $\gamma_p(\varphi_\ep):=
|\varphi_\ep|^{p-2}\varphi_\ep$, and we integrate in time:
using assumption {\bf A2} on the left-hand side 
we obtain
\begin{align*}
    &\frac\tau p\int_\Omega|\varphi_\ep(t)|^p\, \xd x+
    \int_{Q_t}B_\ep(\varphi_\ep)\gamma_p(\varphi_\ep)\,\xd x \xd s
    +c_\psi\int_{Q_t}|\varphi_\ep|^{p}\, \xd x \xd s
    -c_\psi^{-1}\int_{Q_t}|\varphi_\ep|^{p-2}\,\xd x \xd s\\
    &\leq
    \frac\tau p\int_\Omega|\varphi_{0,\ep}|^p\,\xd x+
    \int_{Q_t}(\mu_\ep+\chi\sigma_\ep)\gamma_p(\varphi_\ep)\,\xd x \xd s.
\end{align*}
Now, since $\gamma_p$ is nondecreasing, we have
\begin{align*}
    &\int_{Q_t}B_\ep(\varphi_\ep)\gamma_p(\varphi_\ep)\,\xd x \xd s\\
    &=
    \int_0^t\int_{\Omega\times\Omega}
    J_\ep(x-y)(\gamma_p(\varphi_\ep(s,x))-
    \gamma_p(\varphi_\ep(s,y)))
    (\varphi_\ep(s,x)-\varphi_\ep(s,y))\,\xd x
    \,\xd y\,\xd s\geq0\,.
\end{align*}
Hence, using
H\"older and Young inequalities we infer that,
for every $\delta_1,\delta_2>0$,
\begin{align*}
    &\frac\tau p\sup_{t\in[0,T]}\int_\Omega|\varphi_\ep(t)|^p\,\xd x
    +c_\psi\int_{Q}|\varphi_\ep|^{p}\,\xd x \xd t\\
    &\leq
    \frac\tau p\int_\Omega|\varphi_{0,\ep}|^p\,\xd x+
    c_\psi^{-1}\int_{Q}|\varphi_\ep|^{p-2}\,\xd x \xd t+
    \int_0^T\|\mu_\ep(s)
    +\chi\sigma_\ep(s)\|_{L^p(\Omega)}
    \|\varphi_\ep(s)\|^{p-1}_{L^p(\Omega)}\,\xd s\\
    &\leq\frac\tau p\int_\Omega|\varphi_{0,\ep}|^p\,\xd x+
    \delta_1^{\frac{p}{p-2}}\frac{p-2}{p}\int_Q|\varphi_\ep|^p\,\xd x \xd t
    +\frac{2|Q|}{\delta_1^{p/2}c_\psi^{p/2}p}\\
    &\qquad+
    \frac{p-1}p\delta_2^{\frac{p}{p-1}}
    \sup_{t\in[0,T]}
    \|\varphi_\ep(t)\|_{L^p(\Omega)}^p
    +\frac1{p}\delta_2^p
    \|\mu_\ep+\chi\sigma_\ep\|_{L^1(0,T; L^p(\Omega))}^p\,.
\end{align*}
Hence, using the continuous inclusion $V\embed L^6(\Omega)$
and choosing $\delta_1$ and $\delta_2$ sufficiently small, 
recalling assumption {\bf A7} and the estimates
\eqref{eq:est1}, \eqref{eq:est4}, and \eqref{eq:est5}
we deduce that there exists $M>0$ independent of $\ep$
such that 
\begin{equation}
    \label{eq:est10}
    \|\varphi_\ep\|_{L^\infty(0,T; L^6(\Omega))}^6\leq M\,.
\end{equation}

Let us focus now on the term $R_\ep$. By H\"older 
inequality, the inclusion $V\embed L^6(\Omega)$,
the estimate \eqref{eq:est5}, and assumption {\bf A5}
we immediately have
\begin{align*}
  \|R_\ep\|_{L^2(0,T; L^{6/5}(\Omega))}&\leq
  \|P(\varphi_\ep)\|_{L^\infty(0,T; L^{3/2}(\Omega))}
  \|\sigma_\ep+\chi(1-\varphi_\ep) 
  - \mu_\ep\|_{L^2(0,T; L^6(\Omega))}\\
  &\leq M\left(1 + 
  \|\varphi_\ep\|_{L^\infty(0,T;
  L^4(\Omega))}^4\right)\,,
\end{align*}
so that by \eqref{eq:est10}
we have that 
\begin{equation}
    \label{eq:est11}
    \|R_\ep\|_{L^2(0,T; L^{6/5}(\Omega))}\leq M\,.
\end{equation}
Recalling that $L^{6/5}(\Omega)\embed V^*$
continuously, by comparison in equation \eqref{eq3:nl},
estimate \eqref{eq:est3}, 
and the boundedness of  $n$,
it also follows that 
\begin{equation}
    \label{eq:est11}
    \|\sigma_\ep\|_{H^1(0,T; V^*)}\leq M\,.
\end{equation}

\subsection{Passage to the limit}
By estimates \eqref{eq:est1}--\eqref{eq:est11}
and by the classical weak compactness results, 
we deduce that there exist
\begin{align*}
    &\varphi\in H^1(0,T; H)\cap L^\infty(0,T; L^6(\Omega))\,,
    \qquad\mu\in L^2(0,T; V)\,,\\
    &R\in L^2(0,T; L^{6/5}(\Omega))\,,\\
    &\xi,\nu\in L^2(0,T; H)\,,\\
    &\sigma\in H^1(0,T; V^*)\cap L^\infty(0,T; H)\,,
\end{align*}
with
\begin{align*}
    &\sigma+\chi(1-\varphi) \in L^2(0,T; V)\,,
\end{align*}
and such that, as $\ep\searrow0$ 
(possibly along a not relabelled subsequence),
the following convergences hold:
\begin{align}
    \nonumber \varphi_\ep\wstarto\varphi
    \qquad&\text{in } 
    H^1(0,T; H)\cap L^\infty(0,T; L^6(\Omega))\,,\\
    \label{eq:mu-ep}\mu_\ep\rightharpoonup\mu 
    \qquad&\text{in } L^2(0,T;V)\,,\\
    \nonumber R_\ep\rightharpoonup R
    \qquad&\text{in } L^2(0,T; L^{6/5}(\Omega))\,,\\
    \nonumber \psi'(\varphi_\ep)\rightharpoonup\xi 
    \qquad&\text{in } L^2(0,T;H)\,,\\
     \nonumber B_\ep(\varphi_\ep) \rightharpoonup \nu 
     \qquad&\text{in } L^2(0,T;H)\,,\\
     \nonumber \sigma_\ep\wstarto\sigma
     \qquad&\text{in } 
     H^1(0,T; V^*)\cap L^\infty(0,T; H)\,,\\
     \label{eq:term-ep}\sigma_\ep+\chi(1-\varphi_\ep)\rightharpoonup
     \sigma+\chi(1-\varphi)
     \qquad&\text{in } L^2(0,T; V)\,.
\end{align}
It follows then, by the Aubin-Lions compactness 
results (see \cite[Cor.~4]{simon}), that also
\begin{align*}
  \varphi_\ep\to\varphi \qquad&\text{in } 
  C^0([0,T]; V^*)\,,\\
  \sigma_\ep\to\sigma \qquad&\text{in } 
  C^0([0,T]; V^*)\,,\\
  \sigma_\ep+\chi(1-\varphi_\ep)\to
  \sigma+\chi(1-\varphi)
  \qquad&\text{in } L^2(0,T; H)\,.
\end{align*}
We proceed by proving a stronger convergence 
for $(\varphi_\ep)_\ep$. By using the compactness 
lemma \cite[Lem.~3.4]{DST2} and the 
estimate \eqref{eq:est1}, for 
every $\delta>0$, there exists $c_\delta>0$
such that, for every 
$\ep_1,\ep_2>0$ we have 
\begin{align*}
    &\|\varphi_{\ep_1}-
    \varphi_{\ep_2}\|_{C^0([0,T]; H)}\\
    &\leq \delta \left(
    \|E_{\ep_1}(\varphi_{\ep_1})\|_{L^\infty(0,T)}
    +\|E_{\ep_2}(\varphi_{\ep_2})\|_{L^\infty(0,T)}
    \right)+
    c_\delta\|\varphi_{\ep_1}-
    \varphi_{\ep_2}\|_{C^0([0,T]; V^*)}\\
    &\leq \delta M + c_\delta\|\varphi_{\ep_1}-
    \varphi_{\ep_2}\|_{C^0([0,T]; V^*)}\,.
\end{align*}
Since $\delta>0$ is arbitrary and we already know that 
$\varphi_\ep\to\varphi$ in $C^0([0,T]; V^*)$, we infer that 
\[
  \varphi_\ep\to \varphi \qquad\text{in } 
  C^0([0,T]; H)\,,
\]
which yields, by comparison, that 
\[
  \sigma_\ep\to \sigma \qquad\text{in } 
  L^2(0,T; H)\,.
\]
This strong convergence of $\varphi$ allows us to 
identify the limit $\xi$ by strong-weak closure.
Indeed, we can write
$\psi'=(\psi'+C_\psi I)-C_\psi I$, where
$\psi'+C_\psi I$ is maximal monotone by {\bf A2}
and $C_\psi I$ is linear: hence, the strong-weak 
closure of maximal monotone graphs immediately yields
that 
\[
  \xi=\psi'(\varphi) \quad\text{a.e.~in } Q\,.
\]

Let us focus now on the identification of $\nu$:
more specifically, we need to prove
that $\varphi\in L^\infty(0,T; V)\cap L^2(0,T; W)$
and that $\nu=-\Delta\varphi$.
To this end, we argue as in the works
\cite{DST2, DST}: by the Gamma convergence result
in \cite[Lem.~3.3]{DST2}, 
from the estimate \eqref{eq:est1} 
we directly have that 
$\varphi\in L^\infty(0,T; V)$ and 
\[
  \frac12\int_\Omega|\nabla\varphi(t)|^2
  \leq\liminf_{\ep\searrow0}E_\ep(\varphi_\ep(t))\leq M
  \qquad\text{for a.e.~}t\in(0,T)\,.
\]
Furthermore, as a consequence of \cite[Prop.~2.1]{DST2}
(see also \cite[\S~5]{DST2}), by the estimate
\eqref{eq:est1} 
and the strong convergence of $(\varphi_\ep)_\ep$ 
we have that 
\[
  B_\ep(\varphi_\ep) \rightharpoonup
  B(\varphi) \qquad\text{in } L^2(0,T; V^*)\,,
\]
where $B$ is the weak (variational) formulation
of the negative Laplacian with homogeneous 
Neumann conditions, i.e.
\[
  \langle B(v), \zeta\rangle=
  \int_\Omega\nabla v\cdot\nabla\zeta\,,
  \quad v,\zeta\in V\,\xd x.
\]
Since $B_\ep(\varphi)\rightharpoonup\nu$
in $L^2(0,T; H)$, we deduce that 
\[
  \nu=B(\varphi) \in L^2(0,T; H)\,,
\]
which implies by the classical elliptic regularity 
theory that $\varphi\in L^2(0,T; W)$ and 
$\nu=-\Delta\varphi$, as required.

Let us show now how to pass to the limit
in the terms with the mobilities $m$ and $n$.
Since $m$ and $n$ are bounded and continuous, the 
dominated convergence theorem yields that 
\[
  m(\varphi_\ep)\to m(\varphi)\,, \quad
  n(\varphi_\ep)\to n(\varphi)
  \qquad\text{in } 
  L^r(Q) \quad\forall\,r\in[2,+\infty)\,.
\]
Hence, since also 
\[
  \nabla \mu_\ep\rightharpoonup\nabla\mu\,, \quad
  \nabla(\sigma_\ep-\chi(1-\varphi_\ep))\rightharpoonup
  \nabla(\sigma+\chi(1-\varphi))
  \qquad\text{in } L^2(0,T; L^2(\Omega)^3)\,,
\]
we have that, for every $p\in[1,2)$,
\begin{align*}
    m(\varphi_\ep)\nabla\mu_\ep \rightharpoonup
    m(\varphi)\nabla\mu 
    \qquad&\text{in } L^p(0,T; L^p(\Omega)^3)\,,\\
    n(\varphi_\ep)
    \nabla(\sigma_\ep-\chi(1-\varphi_\ep)) \rightharpoonup
    n(\varphi)\nabla(\sigma-\chi(1-\varphi))
    \qquad&\text{in } L^p(0,T; L^p(\Omega)^3)\,.
\end{align*}
Since $W\embed L^\infty(\Omega)$ continuously,
this implies in particular that, for every $p\in[1,2)$,
\begin{align*}
    -\div(m(\varphi_\ep)\nabla\mu_\ep) \rightharpoonup
    -\div(m(\varphi)\nabla\mu)
    \qquad&\text{in } L^p(0,T; W^*)\,,\\
    -\div(n(\varphi_\ep)
    \nabla(\sigma_\ep-\chi(1-\varphi_\ep))) \rightharpoonup
    -\div(n(\varphi)\nabla(\sigma-\chi(1-\varphi)))
    \qquad&\text{in } L^p(0,T; W^*)\,.
\end{align*}

Finally, the last thing that we need to show is
to identify the term $R$. To this end, 
we note first of all that by the continuous inclusion
$V\embed L^6(\Omega)$, from \eqref{eq:mu-ep} and \eqref{eq:term-ep} we have
\[
  \sigma_\ep + \chi(1-\varphi_\ep)-\mu_\ep
  \rightharpoonup
  \sigma+\chi(1-\varphi) - \mu 
  \qquad\text{in } L^2(0,T; L^6(\Omega))\,.
\]
By the continuity of $P$, 
\[
  P(\varphi_\ep)\to P(\varphi)
  \quad\text{a.e.~in } Q\,,
\]
where, by assumption {\bf A5},
\[
  |P(\varphi_\ep)| \leq M\left(1+|\varphi_\ep|^4\right)\,.
\]
Now,
$(|\varphi_\ep|^4)_\ep$ is bounded in 
$L^\infty(0,T; L^{3/2}(\Omega))$ by \eqref{eq:est10},
hence the
Vitali's dominated convergence theorem yields  
\[
  P(\varphi_\ep)\to P(\varphi) \qquad\text{in }
  L^r(0,T; L^p(\Omega)) \quad\forall\,p\in[1,3/2)\,,
  \quad\forall\,r\in[2,+\infty)\,.
\]
We deduce that, for every $s\in[1,2)$
\[
  R_\ep=P(\varphi_\ep)
  (\sigma_\ep+\chi(1-\varphi_\ep) - \mu_\ep)
  \rightharpoonup P(\varphi)
  (\sigma+\chi(1-\varphi) - \mu) \qquad
  \text{in } L^s(0,T; L^{6/5}(\Omega))\,,
\]
which implies  
\[
  R= P(\varphi)
  (\sigma+\chi(1-\varphi) - \mu)\,.
\]

The information on $\{(u_\ep,w_\ep)\}_\ep$
and assumption {\bf A6} readily imply
by the dominated convergence theorem that 
\begin{align*}
    h(\varphi_\ep)u_\ep\rightharpoonup h(\varphi)u
    \quad\text{in } L^2(0,T; H)\,.
\end{align*}

Letting $\ep\searrow0$ in the weak formulation 
of the nonlocal system in Theorem~\ref{th:wp_nl},
we pass to the limit and obtain exactly 
the variational formulation of the local system 
in Theorem~\ref{th:conv}, for any test function 
$\zeta\in W$. As $W\embed V$ densely, 
the existence part of Theorem~\ref{th:conv} is proved.

\subsection{Uniqueness of the local state system}
Eventually, we show here that under the additional assumptions
{\bf B3}--{\bf B4} the triplet $(\varphi,\mu,\sigma)$ is actually
the unique weak solution of the local system
and there is continuous dependence with respect to the initial data and the controls.
Let us introduce the following notation 
\begin{align*}
&\hat\varphi:=\varphi_1-\varphi_2,\quad \hat\mu:=\mu_1-\mu_2, \quad \hat\sigma:=\sigma_1-\sigma_2, \\
&\hat u:=u_1-u_2, \quad \hat w:=w_1-w_2, \quad \hat\varphi_0:=\varphi_{0}^1-\varphi_{0}^2, \quad \hat\sigma_0:=\sigma_{0}^1-\sigma_{0}^2\,, 
\end{align*}
where $(\varphi_1,\sigma_1,\mu_1)$ and $(\varphi_2,\sigma_2,\mu_2)$ are two solutions of \eqref{loc_last} corresponding to the data $(u_1,w_1)$, $(\varphi_{0}^1, \sigma_{0}^1)$ and $(u_2,w_2)$, $(\varphi_{0}^2, \sigma_{0}^2)$.
Recalling that in this setting both mobilities are constant and equal to $1$ and that $\tau>0$, we rewrite  system \eqref{loc_last} for the differences and we find 
\begin{align}
    \label{loc-diff1}
    &\langle\partial_t\hat\varphi,
      \zeta\rangle
      +\int_\Omega
     \nabla\hat\mu\cdot
      \nabla\zeta\, \xd x = 
      \int_\Omega
     ( R_1-R_2)\zeta\,\xd x
      -\int_\Omega
      ((h(\varphi_1)-h(\varphi_2))u_1+h(\varphi_2)\hat u)\zeta\,\xd x,\\
      \label{loc-diff2}
      &\langle\partial_t\hat\sigma,
      \zeta\rangle +
      \int_\Omega 
      \nabla(\hat\sigma+
      \chi(1-\hat\varphi))\cdot
      \nabla\zeta\,\xd x=
      -\int_\Omega
      (R_1-R_2)\zeta\,\xd x
      +\int_\Omega
      \hat w\zeta\,\xd x,
\end{align}
 almost everywhere in $(0,T)$,
  for every $\zeta\in V$, where 
\begin{align}
    \label{loc-diff-3}
&\hat\mu=
  \tau\partial_t\hat\varphi
  -\Delta\hat\varphi+
  \psi'(\varphi_1)-\psi'(\varphi_2)-\chi\hat\sigma\,,\\
  \label{loc-diff4}
  &R_1-R_2:=(P(\varphi_1)-P(\varphi_2))(\sigma_1+\chi(1-\varphi_1)-\mu_1)+P(\varphi_2)(\hat\sigma+\chi(1-\hat\varphi))\,.
  \end{align}
In the following estimates we denote by $C$ positive constants (which may also differ from line to line) depending on the problem data, and, in particular on $\tau$.
Testing \eqref{loc-diff1} by $\hat\mu$ and \eqref{loc-diff-3}  by $\partial_t\hat\varphi$, summing both expressions, and integrating in time, we infer
\begin{align}\label{u-1}
&\int_0^t\|\nabla\mu\|^2_{H}\,\xd s+\tau\int_0^t\|\partial_t\hat\varphi\|_H^2 \,\xd s+\frac12\|\nabla\hat\varphi(t)\|_H^2\leq \\
\nonumber &\frac12\|\nabla\hat\varphi_0\|_H^2+\int_0^t\int_\Omega (R_1-R_2)\hat\mu\,\xd x\,\xd s+\int_0^t\int_\Omega ((h(\varphi_1)-h(\varphi_2))u_1+h(\varphi_2)\hat u)\hat\mu\,\xd x\,\xd s\\
\nonumber &
+\int_0^t\int_\Omega(\psi'(\varphi_1)-\psi'(\varphi_2))\partial_t\hat\varphi\,\xd x\,\xd s+\chi\int_0^t\int_\Omega\hat\sigma \partial_t\hat\varphi\,\xd x\,\xd s.
\end{align}
Testing \eqref{loc-diff2} by $\hat\sigma+\chi(1-\hat\varphi)$, we deduce
\begin{align}\label{u-2}
&\frac12\|\hat\sigma(t)\|_H^2+\int_0^t\|\nabla(\hat\sigma+\chi(1-\hat\varphi))\|_H^2\,\xd s\leq \frac12 \|\hat\sigma_0\|_H^2\\
\nonumber
&-\int_0^t\int_\Omega\partial_t\hat\sigma(\chi(1-\hat\varphi))\,\xd x\,\xd s+\int_0^t\int_\Omega \left|((R_1-R_2)-\hat w)( \hat\sigma+\chi(1-\hat\varphi))\right|\,\xd x\,\xd s.
\end{align}
In view of \eqref{loc-diff-3}, recalling that $\partial_{\bf n}\hat\varphi=0$ we obtain
\[
\left|\int_\Omega\hat\mu\,\xd x\right|\leq \tau\left|\int_\Omega \partial_t\hat\varphi\,\xd x\right|+\int_\Omega (\psi'(\varphi_1)-\psi'(\varphi_2)-\chi\hat\sigma)\,\xd x.
\]
In particular, taking the square in the previous inequality, using assumption {\bf (A2)} on $\psi$, and integrating the result in time, yields
\begin{align}
\nonumber
\int_0^t\left(\int_\Omega\hat\mu\,\xd x\right)^2\,\xd s&\leq 2\tau^2 |\Omega| \int_0^t\|\partial_t\hat\varphi\|^2_{H}\,\xd s+
C\int_{Q_t}(1+|\varphi_1|^4+|\varphi_2|^4)|\hat\varphi|^2\,\xd x\,\xd s\\
&\quad+C\int_{Q_t} |\hat\sigma|^2\,\xd x\,\xd s\,.
\label{u-3}
\end{align}
We proceed multiplying \eqref{u-3} by a constant $\delta<\min\{1,\frac{1}{4\tau|\Omega|}\}$, and summing the resulting inequality with \eqref{u-1} and \eqref{u-2}: 
\begin{align*}
&\int_0^t\|\nabla\hat\mu\|_H^2\,\xd s +\delta\int_0^t\left(\int_\Omega\hat\mu\,\xd x\right)^2\,\xd s+\frac\tau2\int_0^t\|\partial_t\hat\varphi\|_H^2 \,\xd s+\frac12\|\nabla\hat\varphi(t)\|_H^2 \\
&\qquad+ \frac12\|\hat\sigma(t)\|_H^2
+\int_0^t\|\nabla(\hat\sigma+\chi(1-\hat\varphi))\|_H^2\,\xd s\\
&\leq C\delta
\int_{Q_t}(1+|\varphi_1|^4+|\varphi_2|^4)|\hat\varphi|^2\,\xd x\,\xd s+  C\delta\int_{Q_t}|\hat\sigma|^2\,\xd x\xd s\\
&\qquad+\frac12 \|\hat\sigma_0\|_H^2+\int_\Omega (\hat\sigma(t)\hat\varphi(t)-\hat\sigma_0\hat\varphi_0)\,\xd x\\
&\qquad+\int_0^t \int_\Omega|(R_1-R_2-\hat w))(\hat\sigma+\chi(1-\hat\varphi))|\,\xd x\xd s+\frac12\|\nabla\hat\varphi_0\|_H^2+\int_0^t\int_\Omega (R_1-R_2)\hat\mu\,\xd x\,\xd s\\
\nonumber &\qquad
+\int_0^t\int_\Omega ((h(\varphi_1)-h(\varphi_2))u_1+h(\varphi_2)\hat u)\hat\mu\,\xd x\,\xd s+\int_0^t\int_\Omega(\psi'(\varphi_1)-\psi'(\varphi_2))\partial_t\hat\varphi\,\xd x\,\xd s.
\end{align*}
Adding to both sides of the previous estimate $\frac12\|\hat\varphi\|_H^2$ and in view of assumption {\bf (A2)} on $\psi$, the Poincar\'e-Wirtinger inequality, the H\"older and Young inequalities, and the inclusion $V\embed L^6(\Omega)$, we further infer
\begin{align}
\nonumber
&\frac\delta2\int_0^t\|\hat\mu\|_V^2\,\xd s
+\frac\tau4\int_0^t\|\partial_t\hat\varphi\|_H^2 \,\xd s+\frac12\|\hat\varphi(t)\|_V^2
+ \frac14\|\hat\sigma(t)\|_H^2
+\frac12\int_0^t\|\hat\sigma+\chi(1-\hat\varphi)\|_V^2\,\xd s\\
\nonumber&\leq 
C(1+\|\hat w\|_{L^\infty(Q)})
\int_0^t(1+\|\varphi_1\|_V^4
+\|\varphi_2\|_V^4)\|\hat\varphi\|_V^2\,\xd s
+ C(1+\|\hat w\|_{L^\infty(Q)})\int_{Q_t}|\hat\sigma|^2\,\xd x\xd s\\
\nonumber
&\qquad +C \|\hat\sigma_0\|_H^2+C\|\hat\varphi_0\|_V^2
+C\|\hat\varphi(t)\|_H^2
+C\int_0^t \|R_1-R_2\|^2_{L^{6/5}(\Omega)}\,\xd s\\
\label{laststepuniqueness}&\qquad
+C\int_0^t \|\hat w\|^2_{L^{6/5}(\Omega)}\,\xd s
+\int_0^t\|(h(\varphi_1)-h(\varphi_2))u_1\|_{L^{6/5}(\Omega)}^2\,\xd s
+\int_0^t\|h(\varphi_2)\hat u\|_{L^{6/5}(\Omega)}^2\,\xd s\,.
\end{align}
Now, on the one hand, recalling \eqref{loc-diff4} as well as the Lipschitz regularity of the proliferation function $P$ (see assumption {\bf B4}), by the H\"older inequality and the inclusion $V\embed L^6(\Omega)$
we infer the estimate 
\begin{align}
   \nonumber
   \int_0^t\|R_1-R_2\|^2_{L^{6/5}(\Omega))}\,\xd s&\leq
C\int_0^t\left(1+\|\sigma_1\|_V^2+\|\varphi_1\|_V^2+\|\mu_1\|_V^2\right)
   \|\hat \varphi\|_H^2\,\xd s\\
   \label{stimaR}
   &\quad+C\int_0^t\|\hat\sigma+\chi(1-\hat\varphi)\|_H^2\,\xd s.
\end{align}
On the other hand, by assumption {\bf B4} we also have
\begin{align}
   \nonumber
   &\int_0^t\|(h(\varphi_1)-h(\varphi_2))u_1\|_{L^{6/5}(\Omega)}^2\,\xd s
  +\int_0^t\|h(\varphi_2)\hat u\|_{L^{6/5}(\Omega)}^2\,\xd s\\
  \label{stimah}
  &\leq C\|u_1\|_{L^\infty(Q)}\int_0^t\|\hat\varphi\|_H^2\,\xd s+
  C\int_0^t\|\hat u\|^2_{L^{6/5}(\Omega)}\,\xd s.
\end{align}
By combining \eqref{laststepuniqueness}--\eqref{stimah}, and using the chain rule formula $$\|\hat\varphi(t)\|_H^2={\color{blue}\|\hat\varphi_0\|_H^2+}2\int_0^t(\partial_t\hat\varphi, \hat\varphi)\,\xd s\leq {\color{blue}\|\hat\varphi_0\|_H^2+}\frac\tau8\int_0^t\|\partial_t\hat\varphi\|_H^2\xd s+C\int_0^t\|\hat\varphi\|_H^2\,\xd s\,,$$ 
we get 
\begin{align}
\nonumber
&\int_0^t\|\hat\mu\|_V^2\,\xd s
+\int_0^t\|\partial_t\hat\varphi\|_H^2 \,\xd s+\|\hat\varphi(t)\|_V^2
+ \|\hat\sigma(t)\|_H^2
+\int_0^t\|\hat\sigma+\chi(1-\hat\varphi)\|_V^2\,\xd s\\
\nonumber&\leq 
C\int_0^t\left(\left(1+\|\sigma_1\|_V^2+\|\mu_1\|_V^2\right)\|\hat\varphi\|_V^2
+ \|\hat\sigma\|_H^2\right)\,\xd s+C\|\hat\sigma_0\|_H^2+C\|\hat\varphi_0\|_V^2\\
\nonumber
&\quad +C\int_0^t \left(\|\hat w\|^2_{L^{6/5}(\Omega)}+\|\hat u\|^2_{L^{6/5}(\Omega)}\right)\,\xd s\,.
\end{align}
Finally, the thesis follows by applying the Gronwall Lemma. Uniqueness 
follows by taking $\hat u=\hat w=\hat \varphi_0=\hat \sigma_0=0$.

\subsection{Further properties of the limit system}
The extra regularity $\varphi\in L^2(0,T; H^3(\Omega))$
is a consequence of uniqueness and of \cite[Thm.~1]{frig-grass-roc}.
Eventually, under {\bf B4} one has that $P$ is Lipschitz-continuous so that 
\[
P(\varphi_\ep)\to P(\varphi) \quad\text{in } C^0([0,T]; H)\,.
\]
Moreover, since by the inclusion 
$V\embed L^6(\Omega)$ one has that 
\[
\sigma_\ep+\chi(1-\varphi_\ep)-\mu_\ep
\rightharpoonup
\sigma+\chi(1-\varphi)-\mu \quad\text{in } L^2(0,T; L^6(\Omega))\,,
\]
we readily deduce that 
\[
R_\ep\rightharpoonup R \quad\text{in } L^2(0,T; L^{3/2}(\Omega))\,.
\]
Now, testing \eqref{eq3:nl} by $\sigma_\ep+\chi(1-\varphi_\ep)$ and integrating by parts in time yields, 
for every $t\in[0,T]$, 
\begin{align*}
    &\frac12\Vert\sigma_\ep(t)\Vert_H^2
    +\int_{Q_t}n(\varphi_\ep)|\nabla(\sigma_\ep+\chi(1-\varphi_\ep))|^2
    \,\xd x\xd s\\
    &=\frac12\Vert\sigma_{0,\ep}\Vert_H^2
    +\int_{Q_t}(-R_\ep+w_\ep)(\sigma_\ep+\chi(1-\varphi_\ep))\,\xd x\xd s\\
    &\qquad
    -\chi\int_{Q_t}\sigma_\ep\partial_t\varphi_\ep
    -\chi\int_\Omega \sigma_\ep(t)(1-\varphi_\ep(t))\,\xd x+
    \chi\int_\Omega \sigma_{0,\ep}(1-\varphi_{0,\ep})\,\xd x\,.
\end{align*}
By recalling that $\sigma_\ep(t)\rightharpoonup\sigma(t)$
in $H$, that $\sigma_\ep+\chi(1-\varphi_\ep)\rightharpoonup
\sigma+\chi(1-\varphi)$ in $L^2(0,T; V)$ and, by compactness, also that
$\sigma_\ep+\chi(1-\varphi_\ep)\to
\sigma+\chi(1-\varphi)$ in $L^2(0,T; L^p(\Omega))$ for every $p\in[1,6)$,
by exploiting the convergences obtained above we get 
by weak lower semicontinuity that 
\begin{align*}
  \limsup_{\ep\searrow0}\frac12\Vert\sigma_\ep(t)\Vert_H^2
  &=\lim_{\ep\searrow0}\frac12\Vert\sigma_{0,\ep}\Vert_H^2
  -\liminf_{\ep\searrow0}
  \int_{Q_t}n(\varphi_\ep)|\nabla(\sigma_\ep+\chi(1-\varphi_\ep))|^2
    \,\xd x\xd s\\
  &\qquad+\lim_{\ep\searrow0}\left[
  \int_{Q_t}(-R_\ep+w_\ep)(\sigma_\ep+\chi(1-\varphi_\ep))\,\xd x\xd s
  \right.\\
  &\qquad\left.-\chi\int_{Q_t}\sigma_\ep\partial_t\varphi_\ep
    -\chi\int_\Omega \sigma_\ep(t)(1-\varphi_\ep(t))\,\xd x+
    \chi\int_\Omega \sigma_{0,\ep}(1-\varphi_{0,\ep})\,\xd x\right]\\
  &\leq\frac12\Vert\sigma_{0}\Vert_H^2
  -\int_{Q_t}n(\varphi)|\nabla(\sigma+\chi(1-\varphi))|^2
    \,\xd x\xd s\\
  &\qquad+\int_{Q_t}(-R+w)(\sigma+\chi(1-\varphi))\,\xd x\xd s\\
  &\qquad-\chi\int_{Q_t}\sigma\partial_t\varphi
    -\chi\int_\Omega \sigma(t)(1-\varphi(t))\,\xd x+
    \chi\int_\Omega \sigma_{0}(1-\varphi_{0})\,\xd x\,.
\end{align*}
Now, the analogous computations on the local system imply that 
\begin{align*}
    &\frac12\Vert\sigma(t)\Vert_H^2
    +\int_{Q_t}n(\varphi)|\nabla(\sigma+\chi(1-\varphi))|^2
    \,\xd x\xd s\\
    &=\frac12\Vert\sigma_{0}\Vert_H^2
    +\int_{Q_t}(-R+w)(\sigma+\chi(1-\varphi))\,\xd x\xd s\\
    &\qquad
    -\chi\int_{Q_t}\sigma\partial_t\varphi
    -\chi\int_\Omega \sigma(t)(1-\varphi(t))\,\xd x+
    \chi\int_\Omega \sigma_{0}(1-\varphi_{0})\,\xd x\,,
\end{align*}
so that we deduce 
\[
\limsup_{\ep\searrow0}\frac12\Vert\sigma_\ep(t)\Vert_H^2\leq
\frac12\Vert\sigma(t)\Vert_H^2 \quad\forall\,t\in[0,T]\,,
\]
yielding the last desired convergence 
\[
\sigma_\ep(t)\to \sigma(t) \quad\text{in } H\quad\forall\,t\in[0,T]\,.
\]
This concludes the proof of Theorem~\ref{th:conv}.

\section{The optimal control problem}
\label{sec:opt}
Once the nonlocal-to-local convergence of the state-system is established, 
we are ready to state the optimal control problem.
In this direction, we introduce the following assumptions.
\begin{enumerate}[start=1,label={{\bf C\arabic*}}]
  \item
    The space of admissible controls is defined as
    \begin{align*}
    \mU&:=\left\{(u,w)\in [L^\infty(Q)\cap H^1(0,T; H)]
    \times L^\infty(Q):\right.\\
    &\qquad\left.\;
    u_{min}\leq u\leq u_{max}\,,\quad
    w_{min}\leq w\leq w_{max} \quad\text{a.e.~in } Q\,,
    \right.\\
    &\qquad\left.
    \; \|u\|_{L^\infty(Q)\cap H^1(0,T; H)}
    +\|w\|_{L^\infty(Q)}\leq C_\mU\right\}\,,
    \end{align*}
    where $u_{min}, u_{max}, w_{min}, w_{max}\in L^\infty(Q)$, 
    $u_{min}\geq0$ a.e.~in $Q$, and
    $C_\mU>0$ is a prescribed constant.
  \item The coefficients $\alpha_\Omega, \alpha_Q, \beta_Q, \alpha_u, \beta_w$ are nonnegative but not all zero.
  \item The targets satisfy $\varphi_\Omega\in H$ and
  $\varphi_Q, \sigma_Q\in L^2(0,T; H)$.
  \item The proliferation and truncation functions
  satisfy $P,h\in C^2(\mathbb R)\cap W^{1,\infty}(\mathbb R)$.
\end{enumerate}
In this setting, the
optimisation problem associated to the local state-system reads:\\

\noindent(CP) \textit{Minimise the cost functional}
\begin{equation}
	\begin{split}
		\mathcal{J}(\varphi, \sigma, u, w) & = \, \frac{\alpha_{\Omega}}{2} \int_{\Omega} |\varphi(T) - \varphi_{\Omega}|^2 + \frac{\alpha_Q}{2} \int_{0}^{T} \int_{\Omega} |\varphi - \varphi_Q|^2 
         + \frac{\beta_Q}{2} \int_{0}^{T} \int_{\Omega} |\sigma - \sigma_Q|^2 \\ 
		& \quad + \frac{\alpha_u}{2} \int_{0}^{T} \int_{\Omega} |u|^2  + \frac{\beta_w}{2} \int_{0}^{T} \int_{\Omega} |w|^2\,,
	\end{split}
\end{equation}
\textit{subject to the control constraint $(u,w)\in\mathcal U$,
where $(\varphi,\mu,\sigma)$ is a weak solution
to the local state system associated to $(u,w)$
in the sense of conditions
\eqref{loc_first}--\eqref{loc_last}.}\\

In the setting {\bf A1}--{\bf A7}, {\bf B1}--{\bf B5}, {\bf C1}--{\bf C4},
for every admissible control $(u,w)\in\mU$,
the local state system admits a unique solution $(\varphi,\mu,\sigma)$
in the sense of \eqref{loc_first}--\eqref{loc_last}. Hence, the solution map
\begin{align*}
  \mathcal S:\mU\to 
  &[H^1(0,T; H)\cap L^\infty(0,T; V)\cap L^2(0,T; H^3(\Omega))]\times
  [H^1(0,T; V^*)\cap L^2(0,T; V)]
\end{align*}
is well-defined as
\[
  \mathcal S:(u,w)\mapsto(\varphi,\sigma)\,.
\]
Hence, the optimisation problem (CP) can be written in so-called reduced form
and we define optimal control for problem (CP) any 
$(\overline u, \overline w)\in\mU$ such that 
\[
  \mathcal J(\mathcal S(\overline u, \overline w), \overline u, \overline w)
  =\min_{(u,w)\in\mU}\mathcal J(\mathcal S(u,w), u, w)\,.
\]
Existence of optimal controls for (CP) follows directly by variational arguments.
\begin{proposition}[Existence of optimal controls for (CP)]
    \label{th:ex_opt}
    Assume {\bf A1}--{\bf A7}, {\bf B1}--{\bf B5}, {\bf C1}--{\bf C4}.
    Then, the problem (CP) admits at least an optimal control $(\overline u, 
    \overline w)\in\mathcal U$.
\end{proposition}
\begin{proof}[Proof of Proposition~\ref{th:ex_opt}]
  This is a direct consequence of the direct method 
  of calculus of variations. For a detailed proof in the nonlocal case 
  we refer to \cite[Thm~4.2]{fornoni}: this can be 
  easily adapted to the local case with no additional technical difficulties.
\end{proof}

\begin{remark}
Let us notice that the uniqueness of the optimal control is not to be expected here, because of the nonlinearity of the state system, which destroys the convexity of the problem.
\end{remark}

The main issue on which we focus in this paper is the identification of first 
order conditions for optimality related to optimal controls for (CP).
Instead of tackling the optimisation problem directly
at the local level, 
we exploit a suitable approximation of 
(CP) via some optimisation problems (CP)$_{\ep}$
involving the nonlocal system.
More specifically, (CP)$_\ep$ will be constructed in such a way 
that the optimal controls $(\overline u_\ep)_\ep$ of (CP)$_\ep$
converge to a prescribed optimal control $\overline u$ of (CP).
The necessary conditions for optimality for (CP)$_\ep$
are available in the literature in terms of a suitable dual system:
by proving an analogous nonlocal-to-local convergence result 
for such dual systems, we will then pass to the limit in the first order
conditions for optimality for (CP)$_\ep$ and deduce the 
corresponding ones for (CP).
This will avoid the imposition of extra regularity on
the solutions to the local problem, which would be needed
should one tackle the local optimisation directly via 
differentiability of $\mathcal S$ for example.

\subsection{The adapted optimal control problem}
Let us introduce now the family (CP)$_\ep$ of adapted optimal control problems:
we recall that we work in the setting
{\bf A1}--{\bf A7}, {\bf B1}--{\bf B5}, {\bf C1}--{\bf C4}.

Let $(\overline u, \overline w)\in\mathcal U$
be an optimal control for the problem (CP)
and let $(\overline\varphi, \overline\sigma)
:=\mathcal S(\overline u, \overline w)$ be its corresponding 
optimal state.
For every $\ep\in(0,\ep_0)$, we introduce 
the following adapted optimisation problem.\\

\noindent(CP)$_\ep$ \textit{Minimise the cost functional}
\begin{equation}
	\begin{split}
		\mathcal{J_\ep}(\phi_\ep, \sigma_\ep, u, w) & = 
        \mathcal{J}(\phi_\ep, \sigma_\ep, u, w)
        +\frac12\Vert u-\overline u\Vert_{L^2(0,T; H)}^2
        +\frac12\Vert w-\overline w\Vert_{L^2(0,T; H)}^2
	\end{split}
\end{equation}
\textit{subject to the control constraint $(u,w)\in\mathcal U$,
where $(\varphi_\ep,\mu_\ep,\sigma_\ep)$ is the unique strong 
solution
to the nonlocal state system associated to $(u,w)$
in the sense of conditions
\eqref{nl1_str}--\eqref{nl7_str}.}\\

Since the nonlocal state system is well posed in the sense of Theorem~\ref{wp_nl_strong}, the nonlocal solution 
map 
\begin{align*}
  \mathcal S_\ep:\mU\to 
  &[W^{1,\infty}(0,T; H)\cap H^1(0,T; H^2(\Omega))]
  \times [H^1(0,T; H)\cap L^\infty(0,T; V)\cap L^2(0,T; W)]
\end{align*}
is well-defined as
\[
  \mathcal S_\ep:(u,w)\mapsto(\varphi_\ep,\sigma_\ep)\,.
\]
Hence, also 
the optimisation problem (CP)$_\ep$ can be written in reduced form
and we define optimal control for the problem (CP)$_\ep$ any 
$(\overline u_\ep, \overline w_\ep)\in\mU$ such that 
\[
  \mathcal J_\ep(\mathcal S_\ep(\overline u_\ep, \overline w_\ep), 
  \overline u_\ep, \overline w_\ep)
  =\min_{(u,w)\in\mU}\mathcal J_\ep(\mathcal S_\ep(u,w), u, w)\,.
\]
Note that the existence of optimal controls for (CP)$_\ep$ follows by \cite[Thm. 4.2]{fornoni}. In the theorem below we characterize the local asymptotics for optimal controls associated to (CP)$_\ep$.

\begin{theorem}[Convergence of optimal controls]
    \label{th:opt_adap}
    Assume {\bf A1}--{\bf A7}, {\bf B1}--{\bf B5}, {\bf C1}--{\bf C4}.
    Let $(\overline u, \overline w)\in\mathcal U$ 
    be an optimal control for the problem (CP)
    and let $(\overline\varphi, \overline\sigma)
    :=\mathcal S(\overline u, \overline w)$
    be the corresponding optimal state of the local system.
    Then, for every $\ep\in(0,\ep_0)$
    the adapted optimisation problem (CP)$_\ep$ admits 
    at least an optimal control $(\overline u_\ep, \overline w_\ep)$.
    Moreover, for every family 
    $\{(\overline u_\ep, \overline w_\ep)\}_{\ep\in(0,\ep_0)}$
    of optimal controls for the problems (CP)$_\ep$, 
    if ${(\overline\varphi_\ep, \overline\sigma_\ep)
    :=\mathcal S(\overline u_\ep, \overline w_\ep)}_{\ep\in(0,\ep_0)}$
    are the corresponding optimal states of the nonlocal systems,
    as $\ep\searrow0$ there holds 
    \begin{align*}
    \overline u_\ep \to \overline u \quad&\text{in } 
  L^2(0,T; H)\,,\\
  \overline u_\ep \wstarto \overline u \quad&\text{in } 
  L^\infty(Q)\cap H^1(0,T; H)\,,\\
  \overline w_\ep \to \overline w \quad&\text{in } 
  L^2(0,T; H)\,,\\
  \overline w_\ep \wstarto \overline w \quad&\text{in } 
  L^\infty(Q)\,,\\
    \overline\varphi_\ep\to\overline\varphi
    \quad&\text{in } C^0([0,T], H)\,,\\
    \overline\sigma_\ep\to\overline\sigma
    \quad&\text{in } L^2(0,T, H)\,,\\
    \overline\sigma_\ep(t)\to\overline\sigma(t)
    \quad&\text{in } H\quad\forall\,t\in[0,T]\,.
\end{align*}
\end{theorem}
\begin{proof}[Proof of Theorem~\ref{th:opt_adap}]
Let ${(\overline u_\ep, 
\overline w_\ep)}_{\ep\in(0,\ep_0)}\subset \mU$
be an arbitrary family 
of optimal controls for (CP)$_\ep$, and 
for every $\ep\in(0,\ep_0)$ let 
$(\overline\varphi_\ep, \overline\sigma_\ep)
:=\mathcal S(\overline u_\ep, \overline w_\ep)$
be the corresponding nonlocal optimal state.
By definition of $\mU$, we deduce that there exists some 
$(\tilde u, \tilde w)\in\mU$ such that, as $\ep\searrow0$,
\begin{align*}
  \overline u_\ep \wstarto \tilde u \quad&\text{in } 
  L^\infty(Q)\cap H^1(0,T; H)\,,\\
  \overline w_\ep \wstarto \tilde w \quad&\text{in } 
  L^\infty(Q)\,.
\end{align*}
Moreover, 
setting $(\tilde\varphi, \tilde\sigma):=\mathcal S(\tilde u, \tilde w)$,
by  Theorem~\ref{th:conv} one has also, as
$\ep\searrow0$, that 
\begin{align*}
    \overline\varphi_\ep\to\tilde\varphi
    \quad&\text{in } C^0([0,T], H)\,,\\
    \overline\sigma_\ep\to\tilde\sigma
    \quad&\text{in } L^2(0,T, H)\,,\\
    \overline\sigma_\ep(t)\to\tilde\sigma(t)
    \quad&\text{in } H\quad\forall\,t\in[0,T]\,.
\end{align*}
Let now $(u,w)\in\mathcal U$ be arbitrary and let 
$(\varphi_\ep, \sigma_\ep):=\mathcal S_\ep(u,w)$, 
$(\varphi, \sigma):=\mathcal S(u,w)$
be 
the corresponding nonlocal and local optimal states.
Then, again by Theorem~\ref{th:conv} one has also,
as $\ep\searrow0$, that 
\begin{align*}
    \varphi_\ep\to\varphi
    \quad&\text{in } C^0([0,T], H)\,,\\
    \sigma_\ep\to\sigma
    \quad&\text{in } L^2(0,T, H)\,,\\
    \sigma_\ep(t)\to\sigma(t)
    \quad&\text{in } H\quad\forall\,t\in[0,T]\,.
\end{align*}
By definition of 
optimal control for (CP)$_\ep$ we can write the inequality  
\begin{align}
\nonumber
  &\mathcal J(\overline\varphi_\ep,  \overline\sigma_\ep,
  \overline u_\ep, \overline w_\ep)
  +\frac12\Vert \overline u_\ep-\overline u\Vert_{L^2(0,T; H)}^2
    +\frac12\Vert \overline w_\ep-\overline w\Vert_{L^2(0,T; H)}^2\\
    &\leq
    \mathcal J(\varphi_\ep, \sigma_\ep,
    u, w)
  +\frac12\Vert u-\overline u\Vert_{L^2(0,T; H)}^2
    +\frac12\Vert w-\overline w\Vert_{L^2(0,T; H)}^2
    \qquad\forall\,(u,w)\in\mU\,.
    \label{eq:aux_J}
\end{align}
The convergences above imply then by weak lower semicontinuity that 
\begin{align*}
  &\mathcal J(\tilde\varphi,  \tilde\sigma,
  \tilde u, \tilde w)
  +\frac12\Vert \tilde u-\overline u\Vert_{L^2(0,T; H)}^2
    +\frac12\Vert \tilde w-\overline w\Vert_{L^2(0,T; H)}^2\\
    &\leq
    \mathcal J(\varphi, \sigma,
    u, w)
  +\frac12\Vert u-\overline u\Vert_{L^2(0,T; H)}^2
    +\frac12\Vert w-\overline w\Vert_{L^2(0,T; H)}^2
    \qquad\forall\,(u,w)\in\mU\,.
\end{align*}
Since $(u,w)\in\mU$ is arbitrary, 
choosing now $(u,w)=(\overline u, \overline w)\in\mU$ and 
exploiting that $(\overline u, \overline w)$ is an optimal 
control for (CP) we infer that 
\begin{align*}
  &\mathcal J(\tilde\varphi,  \tilde\sigma,
  \tilde u, \tilde w)
  +\frac12\Vert \tilde u-\overline u\Vert_{L^2(0,T; H)}^2
    +\frac12\Vert \tilde w-\overline w\Vert_{L^2(0,T; H)}^2\\
    &\leq
    \mathcal J(\overline\varphi, \overline\sigma,
    \overline u, \overline w)
    \leq \mathcal J(\tilde\varphi,  \tilde\sigma,
  \tilde u, \tilde w) <+\infty\,,
\end{align*}
from which we deduce that $\tilde u=\overline u$ and $\tilde w=\overline w$,
hence also that $\overline \varphi=\tilde \varphi$, $\overline \sigma=\tilde \sigma$, and
\begin{align*}
  \overline u_\ep \wstarto \overline u \quad&\text{in } 
  L^\infty(Q)\cap H^1(0,T; H)\,,\\
  \overline w_\ep \wstarto \overline w \quad&\text{in } 
  L^\infty(Q)\,.
\end{align*}
It only remains to show the strong convergences 
of $(\overline u_\ep)_\ep$ and $(\overline w_\ep)_\ep$ in $L^2(0,T; H)$.
To this end, choosing $(u,w)=(\overline u, \overline w)$ in
\eqref{eq:aux_J} we infer 
\begin{align*}
  \mathcal J(\overline\varphi_\ep,  \overline\sigma_\ep,
  \overline u_\ep, \overline w_\ep)
  +\frac12\Vert \overline u_\ep-\overline u\Vert_{L^2(0,T; H)}^2
    +\frac12\Vert \overline w_\ep-\overline w\Vert_{L^2(0,T; H)}^2
    \leq
    \mathcal J(\varphi_\ep, \sigma_\ep,
    \overline u, \overline w)\,,
\end{align*}
where now
\begin{align*}
    \varphi_\ep\to\overline\varphi
    \quad&\text{in } C^0([0,T], H)\,,\\
    \sigma_\ep\to\overline\sigma
    \quad&\text{in } L^2(0,T, H)\,,\\
    \sigma_\ep(t)\to\overline\sigma(t)
    \quad&\text{in } H\quad\forall\,t\in[0,T]\,.
\end{align*}
Hence, by letting $\ep\searrow0$ one has first that 
\begin{align*}
  \mathcal J(\overline \varphi, \overline \sigma,
    \overline u, \overline w)
    &\leq\liminf_{\ep\searrow0}
  \mathcal J(\overline \varphi_\ep, \overline \sigma_\ep,
    \overline u_\ep, \overline w_\ep)\\
    &\leq\limsup_{\ep\searrow0}
    \mathcal J(\overline \varphi_\ep, \overline \sigma_\ep,
    \overline u_\ep, \overline w_\ep)\\
    &\leq\lim_{\ep\searrow0}\mathcal J(\varphi_\ep, \sigma_\ep,
    \overline u, \overline w)=
    \mathcal J(\overline \varphi, \overline \sigma,
    \overline u, \overline w)\,,
\end{align*}
and also that 
\begin{align*}
  \mathcal J(\overline \varphi, \overline \sigma,
    \overline u, \overline w)
    &\leq\liminf_{\ep\searrow0}\left[
  \mathcal J(\overline \varphi_\ep, \overline \sigma_\ep,
    \overline u_\ep, \overline w_\ep)
  +\frac12\Vert \overline u_\ep-\overline u\Vert_{L^2(0,T; H)}^2
    +\frac12\Vert \overline w_\ep-\overline w\Vert_{L^2(0,T; H)}^2
    \right]\\
    &\leq\limsup_{\ep\searrow0}\left[
    \mathcal J(\overline \varphi_\ep, \overline \sigma_\ep,
    \overline u_\ep, \overline w_\ep)+
  \frac12\Vert \overline u_\ep-\overline u\Vert_{L^2(0,T; H)}^2
    +\frac12\Vert \overline w_\ep-\overline w\Vert_{L^2(0,T; H)}^2
    \right]\\
    &\leq\lim_{\ep\searrow0}\mathcal J(\varphi_\ep, \sigma_\ep,
    \overline u, \overline w)=
    \mathcal J(\overline \varphi, \overline \sigma,
    \overline u, \overline w)\,,
\end{align*}
yielding 
\[
\lim_{\ep\searrow0}
\mathcal J(\overline \varphi_\ep, \overline \sigma_\ep,
    \overline u_\ep, \overline w_\ep)
=\mathcal J(\overline \varphi, \overline \sigma,
    \overline u, \overline w)
\]
and
\[
\lim_{\ep\searrow0}\left[
\frac12\Vert \overline u_\ep-\overline u\Vert_{L^2(0,T; H)}^2
    +\frac12\Vert \overline w_\ep-\overline w\Vert_{L^2(0,T; H)}^2
\right]=0\,.
\]
This concludes the proof.
\end{proof}

\subsection{First order conditions for the adapted problem}
Let us focus here on the first order conditions for optimality 
for the adapted problem (CP)$_\ep$. Thanks to the results in 
\cite{fornoni}, these can be written in terms of a variational 
inequality involving the solutions of a suitable adjoint, or dual, system.

In this direction, the following result ensures that the dual system 
is indeed well posed.
\begin{theorem}[Well posedness of the nonlocal dual system]
\label{th:wp_dual}
Assume {\bf A1}--{\bf A7}, {\bf B1}--{\bf B5}, {\bf C1}--{\bf C4},
and let $\ep\in(0,\ep_0)$.
Let $(\overline u_\ep, \overline w_\ep)\in\mU$ and let 
$(\overline\varphi_\ep, \overline\sigma_\ep):=
\mathcal S_\ep(\overline u, \overline w)$ be the solution to the nonlocal 
state system.
Then, there exists a unique triplet $(p_\ep, q_\ep, r_\ep)$, with 
\begin{align}
    \label{dual1'}
    &p_\ep\in L^2(0,T; W)\,,\\
    \label{dual2'}
    &q_\ep\in L^2(0,T; H)\,,\\
    \label{dual3'}
    &r_\ep\in H^1(0,T; H)\cap L^\infty(0,T; V)\cap L^2(0,T; W)\,,\\
    \label{dual4'}
    &p_\ep+\tau q_\ep \in H^1(0,T; H)\,,
\end{align}
such that 
\begin{equation}
    \label{dual5'}
    (p_\ep+\tau q_\ep)(T)=\alpha_\Omega(\overline\varphi_\ep(T)-\varphi_\Omega)\,, \qquad
    r_\ep(T)=0\,,
\end{equation}
and
\begin{align}
\nonumber
&-\partial_t(p_\ep+\tau q_\ep)
+B_\ep (q_\ep)+\psi''(\overline\varphi_\ep)q_\ep
+\chi\Delta r_\ep
+\chi P(\overline\varphi_\ep)(p_\ep-r_\ep)\\
\label{eq:dual1}
&\qquad=
P'(\overline\varphi_\ep)
(\overline\sigma_\ep+\chi(1-\overline\varphi_\ep)-\overline\mu_\ep)
(p_\ep-r_\ep)
-h'(\overline\varphi_\ep)\overline u_\ep p_\ep+
\alpha_Q(\overline\varphi_\ep-\varphi_Q)\,,\\
\label{eq:dual2}
&-q_\ep
-\Delta p_\ep+P(\overline \varphi_\ep)(p_\ep-r_\ep)=0\,,\\
\label{eq:dual3}
&-\partial_t r_\ep-\Delta r_\ep-
P(\overline \varphi_\ep)(p_\ep-r_\ep)-\chi q_\ep=
\beta_Q(\overline\sigma_\ep-\sigma_Q)\,.
\end{align}
\end{theorem}
\begin{proof}[Proof of Theorem~\ref{th:wp_dual}]
  This follows directly from \cite[Thm.~4.2]{fornoni}.
\end{proof}

The well posedness of the adjoint system allows then to 
formulate first order conditions for optimality for the
adapted problem (CP)$_\ep$. In this sense, we have the following result.
\begin{theorem}[First order conditions for (CP)$_\ep$]
    \label{th:foc_nl}
    Assume {\bf A1}--{\bf A7}, {\bf B1}--{\bf B5}, {\bf C1}--{\bf C4},
    let $\ep\in(0,\ep_0)$, 
    let $(\overline u, \overline w)\in\mU$
    be an optimal control for (CP),
    let $(\overline u_\ep, \overline w_\ep)\in\mU$
    be an optimal control for (CP)$_\ep$, and
    let $(\overline\varphi_\ep, \overline\sigma_\ep):=
    \mathcal S_\ep(\overline u_\ep, \overline w_\ep)$ 
    be the respective optimal state of the nonlocal system.
    Then, if $(p_\ep, q_\ep, r_\ep)$ is the unique solution 
    to the nonlocal dual system in the sense of 
    \eqref{dual1'}--\eqref{eq:dual3}, it holds
    for every $(u,w)\in\mU$ that 
    \begin{equation}
        \label{eq:foc_nl}
        \int_Q(-h(\overline\varphi_\ep)p_\ep
        +\alpha_u\overline u_\ep + \overline u_\ep-\overline u)
        (u-\overline u_\ep)\,\xd x\xd s
        +\int_Q 
        (r_\ep+\beta_w\overline w_\ep + \overline w_\ep - \overline w)
        (w-\overline w_\ep)\,\xd x\xd s
        \geq0\,.
    \end{equation}
\end{theorem}
\begin{proof}[Proof of Theorem~\ref{th:foc_nl}]
    This follows directly from \cite[Thm.~4.8]{fornoni}.
\end{proof}

\section{Nonlocal-to-local convergence of the optimal control problems}
\label{sec:n-l opt cont}
This section is devoted to proving first order conditions for 
optimality for the limit problem (CP), by passing to the limit as 
$\ep\searrow0$ in the first order conditions for (CP)$_\ep$.
This requires two passages to the limit, one in the dual system
obtained in Theorem~\ref{th:wp_dual} and one in the 
variational inequality of Theorem~\ref{th:opt_adap}.
The two main results that we prove are the following.

\begin{theorem}[Convergence of the dual systems]
    \label{th:conv_dual}
    Assume {\bf A1}--{\bf A7}, {\bf B1}--{\bf B5}, {\bf C1}--{\bf C4},
    let $\{(\overline u_\ep, \overline w_\ep)\}_{\ep\in(0,\ep_0)}
    \subset\mU$ be a family of admissible controls, 
    let $\{(\overline\varphi_\ep,
    \overline\sigma_\ep):=
    \mathcal S_\ep(\overline u_\ep, 
    \overline w_\ep)\}_{\ep\in(0,\ep_0)}$ be the corresponding family 
    of nonlocal states, and let $\{(p_\ep, q_\ep, r_\ep)\}_{\ep\in(0,\ep_0)}$ be the corresponding solutions
    to the nonlocal dual problem in the sense of \eqref{dual1'}--\eqref{eq:dual3}. Let also $(\overline u, \overline w)\in\mU$
    be an admissible control such that, as $\ep\searrow0$,
    \begin{align*}
    \overline u_\ep\wstarto \overline u \quad&\text{in } L^\infty(Q)\cap H^1(0,T; H)\,,\\
    \overline w_\ep\wstarto \overline w \quad&\text{in } L^\infty(Q)\,,
    \end{align*}
    and let 
    $(\overline\varphi, \overline\sigma):=
    \mathcal S(\overline u, \overline w)$ be the corresponding local state.
    Then, there exists a 
    unique triplet $(p, q, r)$, with 
\begin{align}
    \label{dual1'_loc}
    &p\in L^2(0,T; W)\,,\\
    \label{dual2'_loc}
    &q\in L^2(0,T; V)\,,\\
    \label{dual3'_loc}
    &r\in H^1(0,T; H)\cap L^\infty(0,T; V)\cap L^2(0,T; W)\,,\\
    \label{dual4'_loc}
    &p+\tau q \in H^1(0,T; V^*)\cap L^\infty(0,T; H)\,,
\end{align}
such that 
\begin{equation}
    \label{dual5'_loc}
    (p+\tau q)(T)=
    \alpha_\Omega(\overline\varphi(T)-\varphi_\Omega)\,, \qquad
    r(T)=0\,,
\end{equation}
and
\begin{align}
\nonumber
&-\partial_t(p+\tau q)
+B(q)+\psi''(\bar{\varphi})q
+\chi\Delta r
+\chi P(\overline\varphi)(p-r)\\
\label{eq:dual1_loc}
&\qquad=
P'(\overline\varphi)
(\overline\sigma+\chi(1-\overline\varphi)-\overline\mu)
(p-r)
-h'(\overline\varphi)\overline u p+
\alpha_Q(\overline\varphi-\varphi_Q)\,,\\
\label{eq:dual2_loc}
&-q
-\Delta p+P(\overline \varphi)(p-r)=0\,,\\
\label{eq:dual3_loc}
&-\partial_t r-\Delta r-
P(\overline \varphi)(p-r)-\chi q=
\beta_Q(\overline\sigma-\sigma_Q)\,,
\end{align}
where $\bar{\mu}:=\tau\partial_t\bar{\varphi}-\Delta\bar{\varphi}+\psi'(\bar{\varphi})-\chi\bar{\sigma}$. 
Additionally, as $\ep\searrow0$ there holds 
\begin{align}
    \label{conv1_dual}
    p_\ep\to p \quad&\text{in } L^2(0,T; V)\,,\\
    \label{conv2_dual}
    p_\ep\rightharpoonup p \quad&\text{in } L^2(0,T; W)\,,\\
    \label{conv3_dual}
    q_\ep\to q\quad&\text{in } L^2(0,T; H)\,,\\
    \label{conv4_dual}
    r_\ep\to r 
    \quad&\text{in } C^0([0,T]; H)\cap L^2(0,T; V)\,,\\
    \label{conv5_dual}
    r_\ep\wstarto r 
    \quad&\text{in } H^1(0,T; H)\cap L^\infty(0,T; V)\cap L^2(0,T; W)\,,\\
    \label{conv6_dual}
    p_\ep+\tau q_\ep\to p+\tau q
    \quad&\text{in } C^0([0,T]; V^*)\,,\\
    \label{conv7_dual}
    p_\ep+\tau q_\ep\wstarto p+\tau q
    \quad&\text{in } H^1(0,T; V^*)\cap L^\infty(0,T; H)\,.
\end{align}
\end{theorem}

\begin{theorem}[First order conditions for (CP)]
    \label{th:foc_local}
    Assume {\bf A1}--{\bf A7}, {\bf B1}--{\bf B5}, {\bf C1}--{\bf C4},
    let $(\overline u, \overline w)\in\mU$
    be an optimal control for (CP),
    and
    let $(\overline\varphi, \overline\sigma):=
    \mathcal S(\overline u, \overline w)$ 
    be the respective optimal state of the local system.
    Then, if $(p, q, r)$ is the unique solution 
    to the local dual system in the sense of 
    \eqref{dual1'_loc}--\eqref{eq:dual3_loc}, 
    for every $(u,w)\in\mU$ there holds
    \begin{equation}
        \label{eq:foc_loc}
        \int_Q(-h(\overline\varphi)p
        +\alpha_u\overline u)
        (u-\overline u)\,\xd x\xd s
        +\int_Q 
        (r+\beta_w\overline w)
        (w-\overline w)\,\xd x\xd s
        \geq0\,.
    \end{equation}
\end{theorem}

The remaining part of the section is devoted to the proofs of 
Theorems~\ref{th:conv_dual}--\ref{th:foc_local}. For convenience of the reader, we subdivide the proof into steps associated to the
corresponding subsections.

\subsection{Uniform estimates on the nonlocal dual system}
\label{ssec:est_dual}
We test \eqref{eq:dual1} by $p_\ep+\tau q_\ep$, \eqref{eq:dual2} by $C_\tau p_\ep$, and \eqref{eq:dual3} by $-\partial_t r_\ep+C_\tau r_\ep$, where $C_\tau>0$ is a constant, dependent on $\tau$, and big enough, which will be chosen later. 

Fix $t\in (0,T)$. We sum the resulting equations,
integrate in $(t,T)\times \Omega$, and add
to both side of the resulting identity the term $\tau C_\tau\int_t^T\int_\Omega |q_\ep|^2\,\xd x\xd s$: we find
\begin{align*}
&\frac12\int_\Omega |p_\ep(t)+\tau q_\ep(t)|^2\,\xd x
+2\tau\int_t^T E_\ep(q_\ep)\,\xd s
+C_\tau \int_t^T \int_\Omega (|\nabla p_\ep|^2
+P(\bar{\varphi}_\ep)|p_\ep|^2)\,\xd x\xd s\\
&\quad+\int_t^T \int_\Omega |\partial_t r_\ep|^2\,\xd x\xd s
+\frac{1}2\int_\Omega \left(C_\tau|r_\ep(t)|^2
+|\nabla r_\ep(t)|^2\right)\,\xd x
+C_\tau \int_t^T\int_\Omega|\nabla r_\ep|^2 \,\xd x\xd s\\
&\quad+\tau C_\tau\int_t^T \int_\Omega |q_\ep|^2\,\xd x\xd s
\\
&= \frac12\int_\Omega\alpha_\Omega^2 
|\overline\varphi_\ep(T)-\varphi_\Omega|^2\,\xd x
-\int_t^T\int_\Omega B_\ep(q_\ep)p_\ep \,\xd x\xd s
-\int_t^T\int_\Omega \psi''(\overline\varphi_\ep)
q_\ep(p_\ep+\tau q_\ep)\,\xd x\xd s\\
&\quad-\int_t^T\int_\Omega \chi \Delta r_\ep 
(p_\ep+\tau q_\ep)\,\xd x\xd s
-\int_t^T\int_\Omega \chi P(\overline\varphi_\ep)
(p_\ep-r_\ep)(p_\ep+\tau q_\ep)\,\xd x\xd s\\
&\quad+\int_t^T\int_\Omega P'(\overline\varphi_\ep)
(\overline\sigma_\ep+\chi(1-\overline\varphi_\ep)-\overline\mu_\ep)
(p_\ep-r_\ep)
(p_\ep+\tau q_\ep)\,\xd x\xd s\\
&\quad-\int_t^T h'(\overline\varphi_\ep)
\overline u_\ep p_\ep(p_\ep+\tau q_\ep)\,\xd x\xd s
+\int_t^T\int_\Omega \alpha_Q
(\overline\varphi_\ep-\varphi_Q)(p_\ep+\tau q_\ep)\,\xd x\xd s\\
&\quad
+C_\tau \int_t^T\int_\Omega P(\overline\varphi_\ep)
r_\ep p_\ep\,\xd x\xd s
+C_\tau\int_t^T\int_\Omega q_\ep(p_\ep+\tau q_\ep)\,\xd x\xd s\\
&\quad +\int_t^T\int_\Omega P(\overline\varphi_\ep)(p_\ep-r_\ep)
(-\partial_t r_\ep+C_\tau r_\ep)\,\xd x\xd s
+\int_t^T\int_\Omega \chi q_\ep(-\partial_t r_\ep+C_\tau r_\ep)\,\xd x\xd s
\\
&\quad
+\int_t^T\int_\Omega \beta_Q(\overline\sigma_\ep-\sigma_Q)
(-\partial_t r_\ep+C_\tau r_\ep)\,\xd x\xd s=:\sum_{i=1}^{13}I_i\,.
\end{align*}
We estimate the 13 terms on the right-hand side separately:
we will use the symbol $C$ for 
a positive constant independent of $\ep$, $\tau$, and $C_\tau$.
Clearly, $I_1$ is bounded since $\overline\varphi_\ep(T)\to\overline\varphi(T)$ in $H$ by \ref{th:conv}.
As for $I_2$, we have
\begin{align}
\label{eq:bd5-dual}
\notag
I_2&=-\int_t^T \int_\Omega B_\ep(q_\ep)p_\ep\,\xd x\xd s\leq 
2\int_t^T\sqrt{E_\ep(q_\ep)}\sqrt{E_\ep(p_\ep)}\,\xd s\\
\notag
&\leq
\frac{\tau}{4}\int_t^T E_\ep(q_\ep)\,\xd s
+\frac{4}{\tau}\int_t^T E_\ep(p_\ep)\,\xd s\\
&\leq \frac{\tau}{4}\int_t^T E_\ep(q_\ep)\,\xd s+
\frac{C}\tau\int_t^T\int_\Omega |\nabla p_\ep|^2\,\xd x\xd s\,.
\end{align}
By \eqref{eq:est1} and the inclusion 
$V\embed L^6(\Omega)$ we infer
\begin{align}
\label{eq:bd6-dual}
\notag I_3=-\int_t^T\int_\Omega \psi''(\bar{\varphi}_\ep)q_\ep p_\ep\,\xd x\xd s
&\leq C\int_t^T
(1+\|\bar{\varphi}_\ep\|^2_{L^6})\|p_\ep\|_{L^6}\|q_\ep\|_{L^2}\,\xd s\\
&\leq C\left(\int_t^T\|q_\ep\|_H^2\xd s + 
\int_t^T\|p_\ep\|_V^2\,\xd s\right).
\end{align}
As for $I_4$, by \eqref{eq:dual3} and the boundedness of $P$ we deduce
\begin{align}
\label{eq:bd7-dual}
\notag
I_4&=-\int_t^T\int_\Omega \chi \Delta r_\ep (p_\ep+\tau q_\ep)\,\xd x\xd s\\
\notag
&\leq 
\chi\int_t^T \int_\Omega |\partial_t r_\ep+P(\bar{\varphi}_\ep)(p_\ep-r_\ep)+\chi q_\ep||p_\ep+\tau q_\ep|\,\xd x\xd s\\
\notag
&\leq\frac14\int_t^T \int_\Omega |\partial_t r_\ep|^2\,\xd x\xd s
+C\int_t^T \int_\Omega (|p_\ep|^2+|r_\ep|^2+|q_\ep|^2)\,\xd x\xd s\\
&\qquad+C\int_t^T \int_\Omega |p_\ep+\tau q_\ep|^2\,\xd x\xd s
\end{align}
and similarly
\begin{align}
\label{eq:bd9-dual}
I_5\leq
C\int_t^T \int_\Omega (|p_\ep|^2+|r_\ep|^2)\,\xd x\xd s
+C\int_t^T \int_\Omega |p_\ep+\tau q_\ep|^2\,\xd x\xd s\,.
\end{align}
As for $I_6$, the boundedness of $P'$ yields, together with
H\"older and Young inequalities,
\begin{align}
\label{eq:bd8-dual}
\notag
I_6&\leq C\int_t^T\int_\Omega 
|\overline\sigma_\ep+
\chi(1-\overline\varphi_\ep)-
\overline\mu_\ep||p_\ep-r_\ep||p_\ep+\tau q_\ep|\,\xd x\xd s\\
\notag
&\leq C \|p_\ep-r_\ep\|_{L^2(t,T;L^3(\Omega))}^2
+C\int_t^T\|\bar{\sigma}_\ep+\chi(1-\bar{\varphi}_\ep)-\bar{\mu}_\ep\|_{L^6(\Omega)}^2\|p_\ep+\tau q_\ep\|_H^2\,\xd s\\
\notag
&\leq C\left(\int_t^T\|p_\ep\|_V^2\xd s+ 
\int_t^T\|r_\ep\|_V^2\,\xd s\right) \\
&\qquad 
+C\int_t^T\left(1+\|\overline\sigma_\ep+\chi(1-\overline\varphi_\ep)\|_V^2+
\|\overline\mu_\ep\|_V^2\right)
\|p_\ep+\tau q_\ep\|_H^2\,\xd s\,.
\end{align}
In view of the boundedness of $h'$ and $P$, the definition of $\mU$, and
the boundedness of $(\overline\varphi_\ep)_\ep$ in $L^2(Q)$ we find
\begin{align} 
\label{eq:bd10-dual}
    \notag
    I_7+I_8+I_9+I_{10}
    &\leq C\left(1+\int_t^T\|p_\ep\|_H^2\xd s+ 
    \int_t^T\|q_\ep\|_H^2\xd s\right)\\
&+CC_\tau^2\int_t^T\left(
\|r_\ep\|_H^2 + \|p_\ep+\tau q_\ep\|_H^2\right)\,\xd s\,.
\end{align}
Similarly, it is immediate to check that 
\begin{align} 
\label{eq:bd11-dual}
    \notag
    I_{11}+I_{12}+I_{13}
    &\leq \frac14\int_t^T \int_\Omega |\partial_t r_\ep|^2\,\xd x\xd s+
    C\left(1+\int_t^T\|p_\ep\|_H^2\xd s+
    \int_t^T\|q_\ep\|_H^2\xd s\right)\\
    &+CC_\tau^2\int_t^T\|r_\ep\|_H^2\,\xd s\,.
\end{align}
Hence, by taking \eqref{eq:bd5-dual}--\eqref{eq:bd11-dual} into account, 
we can choose and fix $C_\tau>0$ large enough so that 
$\tau C_\tau>>C$ and $C_\tau>>C$: recalling that $P$ is bounded from
below by a positive constant, all the terms in
\eqref{eq:bd5-dual}--\eqref{eq:bd11-dual} not associated with $C_\tau$
can be reabsorbed in the left-hand side. By suitably renominating 
the constants, we are then left with 
\begin{align*}
&\|p_\ep(t)+\tau q_\ep(t)\|_H^2 + \|r_\ep(t)\|^2_V
+\int_t^T (\|q_\ep\|_H^2+ E_\ep(q_\ep))\,\xd s
+\int_t^T \|p_\ep\|_V^2\,\xd s\\
&\qquad
+\int_t^T (\|\partial_t r_\ep\|_H^2+\|r_\ep\|_V^2)\,\xd s\\
&\leq \tilde C\left(1+\int_t^T\|r_\ep\|_H^2\,\xd s
+\int_t^T\left(1+\|\overline\sigma_\ep+\chi(1-\overline\varphi_\ep)\|_V^2+
\|\overline\mu_\ep\|_V^2\right)
\|p_\ep+\tau q_\ep\|_H^2\,\xd s\right)\,,
\end{align*}
where $\tilde C$ is a positive constant independent of $\ep$
(but depending on $\tau$).
An application of Gronwall's lemma and a comparison with
\eqref{conv_loc_first}--\eqref{conv_loc_last} yields the uniform estimates
\begin{align}
&\label{est-dual1} \|p_\ep\|_{L^2(0,T; V)}\leq M\,,\\
&\label{est-dual2}
\|p_\ep+\tau q_\ep\|_{L^\infty(0,T;H)}\leq M\,,\\
&\label{est-dual3}
\|q_\ep\|_{L^2(0,T;H)}+\|E_\ep(q_\ep)\|_{L^1(0,T)}\leq M\,,\\
&\label{est-dual4}
\|r_\ep\|_{H^1(0,T;H)\cap L^\infty(0,T;V)}\leq M\,,
\end{align}
where the constant $M$ is independent of $\ep$.
In particular, \eqref{est-dual2}, 
the boundedness of $(\overline\varphi_\ep)_\ep$
in $L^\infty(0,T; L^{6}(\Omega))$, and the inclusion 
$L^{6/5}(\Omega)\embed V^*$ imply that 
\begin{equation}
    \label{est-dual4}
    \|B_\ep (q_\ep)\|_{L^2(0,T; V^*)}+
    \|\psi''(\overline\varphi_\ep) q_\ep\|_{L^2(0,T; V^*)}\leq M\,.
\end{equation}
Hence a direct comparison in \eqref{eq:dual1}--\eqref{eq:dual3} yields also
\begin{align}
&\label{est-dual5} 
\|p_\ep\|_{L^2(0,T; W)}\leq M\,,\\
&\label{est-dual6} 
\|p_\ep+\tau q_\ep\|_{H^1(0,T;V^*)}\leq M\,,\\
&\label{est-dual7}
\|r_\ep\|_{L^2(0,T;W)}\leq M\,.
\end{align}

\subsection{Passage to the limit in the dual system}
The estimates \eqref{est-dual1}--\eqref{est-dual7} imply the existence 
of a triplet $(p,q,r)$ satisfying \eqref{dual1'_loc}--\eqref{dual5'_loc},
as well as, as $\ep\searrow0$ along a subsequence, the convergences 
\eqref{conv1_dual}--\eqref{conv2_dual} and 
\eqref{conv4_dual}--\eqref{conv7_dual}. As for \eqref{conv3_dual},
we note that by comparison one obtains first 
that $q_\ep\to q$ in $L^2(0,T; V^*)$, and this implies 
also $q_\ep\to q$ in $L^2(0,T; H)$ by the compactness lemma \cite[Lem.~3.4]{DST2}. Moreover, from the strong convergence \eqref{conv3_dual}
and the estimate \eqref{est-dual5}, 
thanks to \cite[Prop.~3.1]{DST2} 
(in a suitable integrated-in-time formulation) we deduce that 
$q\in L^2(0,T; V)$ and that 
\[
  B_\ep (q_\ep)\rightharpoonup B(q) \quad\text{in } L^2(0,T; V^*)\,.
\]
Since $P'$ and $h'$ are bounded and continuous, it is immediate to check 
that all the terms in \eqref{eq:dual1}--\eqref{eq:dual3}
pass to the limit as $\ep\searrow0$, and \eqref{eq:dual1_loc}--\eqref{eq:dual3_loc} follow.
This concludes the proof of existence and convergence in Theorem~\ref{th:conv_dual}.
We are only left with showing that 
the solution to the local dual system is unique: to this end, 
it is enough to prove that if $(p,q,r)$ is a solution to the homogeneous 
problem, i.e.~\eqref{dual1'_loc}--\eqref{eq:dual3_loc} with 
formally $\alpha_\Omega=\alpha_Q=\beta_Q=0$, then $p=q=r=0$.
Now, it is straightforward to check that the same estimate performed 
on the nonlocal dual system in Subsection~\ref{ssec:est_dual}
can be repeated in this context with elementary adaptations:
this yields 
\begin{align*}
&\|p(t)+\tau q(t)\|_H^2 + \|r(t)\|^2_V
+\int_t^T \left(\|q\|_H^2+ \|\nabla q\|_H^2
+\|p\|_V^2+\|\partial_t r\|_H^2+\|r\|_V^2\right)\,\xd s
\\
&\leq \tilde C\left(\int_t^T\|r\|_H^2\,\xd s
+\int_t^T\left(1+\|\overline\sigma
+\chi(1-\overline\varphi)\|_V^2+
\|\overline\mu\|_V^2\right)
\|p+\tau q\|_H^2\,\xd s\right)\,,
\end{align*}
so that the Gronwall lemma implies that $p=q=r=0$.
This concludes the proof of Theorem~\ref{th:conv_dual}.

\subsection{Passage to the limit in the first order conditions}
We are ready to prove now Theorem~\ref{th:foc_local}.
Given an optimal control $(\overline u, \overline w)$ for (CP), 
by Theorem~\ref{th:opt_adap} there is a sequence of 
optimal controls $(\overline u_\ep, \overline w_\ep)_\ep$
satisfying the convergences in Theorem~\ref{th:opt_adap}.
Moreover, denoting by $(p_\ep, q_\ep, r_\ep)_\ep$
the solutions of the corresponding nonlocal dual systems, one has 
the convergences \eqref{conv1_dual}--\eqref{conv7_dual} by Theorem~\ref{th:conv_dual}, as well as the first order conditions 
for optimality for (CP)$_\ep$ in \eqref{eq:foc_nl}.
Since $\overline u_\ep\to \overline u$ and 
$\overline w_\ep\to\overline w$ in $L^2(0,T; H)$, by letting 
$\ep\searrow0$ in \eqref{eq:foc_nl} one readily obtains 
\eqref{eq:foc_loc} by the dominated convergence theorem.
This concludes the proof of Theorem~\ref{th:foc_local}.

\section*{Acknowledgements}

E.D. and L.T. have been supported by the Austrian Science Fund (FWF) project F 65. E.D. has been funded by the Austrian Science Fund (FWF) projects V 662, I 4052, Y 1292, and P35359, as well as from BMBWF through project CZ 09/2023. L.S. has been funded by the
Austrian Science Fund (FWF) project M 2876-N.
E.R.~has been supported by the MIUR-PRIN Grant 2020F3NCPX 
``Mathematics for industry 4.0 (Math4I4)''. The present paper also benefits from the support of 
the GNAMPA (Gruppo Nazionale per l'Analisi Matematica, la Probabilit\`a e le loro Applicazioni)
of INdAM (Istituto Nazionale di Alta Matematica).
The present research has also been supported by the 
MUR grant “Dipartimento di eccellenza 2023–2027” for
Politecnico di Milano.

\bibliographystyle{abbrv}
\bibliography{ref}
\end{document}